\documentclass[12pt]{article}
\usepackage{amsmath,amssymb,amscd,amsthm,esint}
\usepackage{graphics,amsmath,amssymb,amsthm,mathrsfs}
\usepackage{amssymb,amsfonts}
\usepackage[all,arc]{xy}
\usepackage{enumerate}
\usepackage{mathrsfs}
\usepackage{amsmath,cite,amsthm}

\newtheorem{theorem}{Theorem}[section]
\newtheorem{prop}[theorem]{Proposition}
\newtheorem{lemma}[theorem]{Lemma}

\newtheorem{remark}[theorem]{Remark}

\newcommand{\average}{-\!\!\!\!\!\!\int}

\begin{document}

\title
{\bf Convergence Rates of Neumann problems for Stokes Systems}

\author{Shu Gu\footnote{Supported in part by NSF grant DMS-1161154}}

\date{ }

\maketitle

\begin{abstract}
In quantitative homogenization of the Neumann problems for Stokes systems with rapidly oscillating periodic coefficients, this paper studies the convergence rates of the velocity in $L^2$ and $H^1$ as well as those of the pressure term in $L^2$ , without any smoothness assumptions on the coefficients.

%\bigskip

\end{abstract}


%\textit{AMS Subject Classifications.} 35B27 

\medskip

%\tableofcontents

%%%%%%%%%%%%%%%%%%%%%%%%%%%%%%%%%%%%%%%%%%%%%%%%%%%%%%%%%%%%%%%%%%%%%%%%%%%%%%%%%%%%%%%%%%%%%%%%%%%%%%%%%%%%%%%%%%%%%%%%%%%%%%%%%%%%%%%%%%
\section{Introduction and Main Results}
\setcounter{equation}{0}

\quad\  In this paper, we would like to investigate the convergence rates of Neumann problems for Stokes systems with rapidly oscillating periodic coefficients. Specifically, we'd like to consider the following Neumann problem for Stokes system in a bounded domain $\Omega\subset \mathbb{R}^d$,
\begin{equation}\label{def.Stokes.Neumann}
\left\{
\begin{aligned}
\mathcal{L}_\varepsilon(u_\varepsilon)+\nabla p_\varepsilon &= F &\qquad \text{ in }\Omega,\\
\text{div}(u_\varepsilon) &=g &\qquad\text{ in }\Omega,\\
\frac{\partial u_\varepsilon}{\partial \nu_\varepsilon}-p_\varepsilon\, n & =f &\qquad\text{ on }\partial\Omega,
\end{aligned}
\right.	
\end{equation}
where $n$ denotes the outward unit normal to $\partial \Omega$. Throughout this paper, we use the summation convention and let $\varepsilon >0$ be a small parameter. We define the second-order elliptic operator in divergence form $\mathcal{L}_\varepsilon$ associated with coefficient matrix $A$ by
\begin{equation}\label{def.operator}
\mathcal{L}_\varepsilon= -\text{div}(A(x/\varepsilon)\nabla)=-\frac{\partial}{\partial x_i}\bigg[a_{ij}^{\alpha\beta}\big(\frac{x}{\varepsilon}\big)\frac{\partial}{\partial x_j}\bigg]
\end{equation}
where $1\le i,j,\alpha,\beta \le d$, and the conormal derivative of system (\ref{def.Stokes.Neumann}) on $\partial\Omega$ is defined by
\begin{equation}\label{def.conormal}
\left(\frac{\partial u_\varepsilon}{\partial \nu_\varepsilon}\right)^\alpha-p_\varepsilon\, n_\alpha=n_i(x)a_{ij}^{\alpha\beta}(x/\varepsilon)\frac{\partial u_\varepsilon^\beta}{\partial x_j}-p_\varepsilon(x)n_\alpha(x).
\end{equation}
We assume that the coefficient matrix $A(y)=(a_{ij}^{\alpha\beta}(y))$ is real, bounded measurable, and it satisfies the ellipticity condition:
\begin{equation}\label{cond.ellipticity}
\mu |\xi|^2 \le a_{ij}^{\alpha\beta}(y)\xi_i^\alpha\xi_j^\beta \le \frac{1}{\mu} |\xi|^2 \qquad \text{for }y\in \mathbb{R}^d \text{ and }\xi=(\xi_i^\alpha) \in \mathbb{R}^{d\times d}, 	
\end{equation} 
where $\mu>0$, and also the periodicity condition,
\begin{equation}\label{cond.periodicity}
A(y+z)=A(y) \qquad \text{ for }y\in \mathbb{R}^d\text{ and }z\in \mathbb{Z}^d.
\end{equation}
A function satisfying (\ref{cond.periodicity}) will be called 1-periodic.

The homogenization theory of Neumann problems for Stokes systems tells us that, $u_\varepsilon-\average_\Omega u_\varepsilon$ converges to  $u_0-\average_\Omega u_0$ weakly in $H^1$, and $p_\varepsilon-\average_\Omega p_\varepsilon$ converges to $p_0-\average_\Omega p_0$ weakly in $L^2$, given suitable $F$, $f$ and $g$.  Here $(u_0, p_0)\in H^1(\Omega;\mathbb{R}^d)\times L^2(\Omega)$ is the weak solution of the associated homogenized problem with constant coefficients,
\begin{equation}\label{def.Stokes0.Neumann}
\left\{
\begin{aligned}
\mathcal{L}_0(u_0)+\nabla p_0 &= F &\qquad \text{ in }\Omega,\\
\text{div}(u_0) &=g &\qquad\text{ in }\Omega,\\
\frac{\partial u_0}{\partial \nu_0}-p_0\, n & =f &\qquad\text{ on }\partial\Omega.
\end{aligned}
\right.	
\end{equation}
The nature and primary question will be how fast does it converge. Our main purpose is to study the optimal convergence rate of $\|u_\varepsilon-u_0\|_{L^2(\Omega)}$, as $\varepsilon \rightarrow 0$. The result is given in the following theorem.
\begin{theorem}\label{thm.convergence-L2.Neumann}
Let $\Omega$ be a bounded $C^{1,1}$ domain. Suppose $A$ satisfies ellipticity condition (\ref{cond.ellipticity}) and periodicity condition (\ref{cond.periodicity}). Given $F\in L^2(\Omega;\mathbb{R}^d)$ and $f\in H^{1/2}(\partial\Omega;\mathbb{R}^d)$ satisfying the compatibility condition
\begin{equation}\label{cond.compatibility.Neumann}
\int_\Omega F +\int_{\partial \Omega}f =0,
\end{equation}
for any $g\in H^1(\Omega)$, let $(u_\varepsilon,p_\varepsilon)$, $(u_0,p_0)$ be the weak solutions of Neumann problems (\ref{def.Stokes.Neumann}), (\ref{def.Stokes0.Neumann}), respectively. If $\int_\Omega u_\varepsilon=\int_\Omega u_0=0$, then
\begin{equation}\label{ineq.convergence-L2.Neumann}
\|u_\varepsilon-u_0\|_{L^2(\Omega)} \le C\varepsilon\|u_0\|_{H^2(\Omega)},
\end{equation}
where the constant $C$ depends only on $\mu$, $d$, and $\Omega$.
\end{theorem}

%Theorem \ref{theorem1.2} gives the optimal $O(\varepsilon)$ convergence 
%rate for the inverses of the Stokes operators in $L^2$ operator norm.
%Indeed, let $T_\varepsilon: F\in L^2_\sigma (\Omega)
%\to u_\varepsilon$, where $L^2_\sigma(\Omega)=
%\big\{ F\in L^2(\Omega; \mathbb{R}^d):\, \text{div} (F)=0   \text{ in } \Omega \big\}$,
%and $u_\varepsilon$ denotes the solution of (\ref{NeumannStokes}) with $F\in L^2_\sigma(\Omega; \mathbb{R}^d)$
%and $g=0$, $f=0$. Then it follows from (\ref{ethm1.2}) and the estimate 
%$\| u_0\|_{H^2(\Omega)} \le C \| F\|_{L^2(\Omega)}$ that
%$$
%\| T_\varepsilon -T_0 \|_{L^2_\sigma (\Omega) \to L^2_\sigma (\Omega)} \le C\varepsilon.
%$$

Theorem \ref{thm.convergence-L2.Neumann} gives us the order $O(\varepsilon)$ convergence of the velocity in $L^2$, which is optimal in the sense of $\|u_0\|_{H^2(\Omega)}$. The other important result of this paper, which is shown in the next theorem, is that the two-scale expansion of $(u_\varepsilon, p_\varepsilon)$ has optimal $O(\varepsilon^{1/2})$ rates in $H^1\times L^2$. 

For simplicity, we will use the notation $h^\varepsilon(x)=h(x/\varepsilon)$, for any function $h$. Here $(\chi,\pi)$ are the correctors associated with $A$, defined as in (\ref{def.cell-problem}), and $S_\varepsilon$ is the Steklov smoothing operator introduced in (\ref{def.Steklov}). 

\begin{theorem}\label{thm.convergence-H1.Neumann}
Let $\Omega$ be a bounded $C^{1,1}$ domain. Suppose $A$ satisfies ellipticity condition (\ref{cond.ellipticity}) and periodicity condition (\ref{cond.periodicity}). Let $(u_\varepsilon,p_\varepsilon)$ and $(u_0,p_0)$ be the same as in Theorem \ref{thm.convergence-L2.Neumann}.  If $\int_\Omega u_\varepsilon=\int_\Omega u_0=0$, then
\begin{equation}\label{ineq.convergence-H1-u.Neumann}
\|u_\varepsilon-u_0-\varepsilon \chi^\varepsilon S_\varepsilon \left(\nabla\widetilde{u}_0\right)\|_{H^1(\Omega)} \le C\sqrt{\varepsilon}\|u_0\|_{H^2(\Omega)},
\end{equation}
where $\widetilde{u}_0$ is the extension of $u_0$ defined as in (\ref{def.extension-operator}). Moreover, if $\int_\Omega p_\varepsilon=\int_\Omega p_0=0$, then
\begin{equation}\label{ineq.convergence-H1-p.Neumann}            
\|p_\varepsilon-p_0-\big[\pi^\varepsilon S_\varepsilon \left(\nabla\widetilde{u}_0\right)-\average_\Omega \pi^\varepsilon S_\varepsilon \left(\nabla\widetilde{u}_0\right)\big]\|_{L^2(\Omega)}\le C\sqrt{\varepsilon}\|u_0\|_{H^2(\Omega)}.
\end{equation}
The constants $C$ in (\ref{ineq.convergence-H1-u.Neumann}) and (\ref{ineq.convergence-H1-p.Neumann}) depend only on $\mu$, $d$, and $\Omega$.
\end{theorem}

There are relatively fewer known $L^2$ convergence rates results for  Neumann problems than for Dirichlet cases. For the scalar elliptic equation $\mathcal{L}_\varepsilon(u_\varepsilon)=-\text{div}(A(x/\varepsilon)\nabla u_\varepsilon)=F$ in $\Omega$ with Neumann condition $\frac{\partial u_\varepsilon}{\partial \nu_\varepsilon}=0$ on $\partial\Omega$, the estimate $\|u_\varepsilon-u_0\|_{L^2(\Omega)}\le C\varepsilon\|F\|_{H^2(\Omega)}$ 
was proved by Griso\cite{Griso06} for $C^{1,1}$ domains with bounded measurable coefficients, by using the method of periodic unfolding (see \cite{CioranescuDamlamianGriso02,CioranescuDamlamianGriso08}). The same result was also proved by Moskow and Vogelius \cite{MoskowVogelius97} for curvilinear convex polygons $\Omega$ in $\mathbb{R}^2$. 

For the system case, consider the standard second-order elliptic systems $\mathcal{L}_\varepsilon (u_\varepsilon)=F$ in $\Omega$ with Neumann condition $\frac{\partial u_\varepsilon}{\partial \nu_\varepsilon}=g$ on $\partial\Omega$, Kenig, Lin and Shen\cite{KenigLinShen12} have shown that the better estimate (\ref{ineq.convergence-L2.Neumann}) holds in bounded Lipschitz domain $\Omega$, under additional assumption that $A$ is H\"older continuous. Let
$$v_\varepsilon=u_\varepsilon-u_0-\varepsilon\chi^\varepsilon \nabla u_0$$ be the difference between $u_\varepsilon$ and its first order approximation. The approach used in \cite{KenigLinShen12} was based on the explicit computation of conormal derivative of $v_\varepsilon$ and uniform regularity estimates for the $L^2$ Neumann problem derived in \cite{KenigShen1102,KenigShen1101}. Moreover, if additionally assume that $A$ is H\"older continuous and symmetric, it was proved in \cite{KenigLinShen12} that $\|v_\varepsilon\|_{H^{1/2}(\partial\Omega)}\le C\varepsilon\|u_0\|_{H^2(\Omega)}$. 

Nevertheless, the method used in \cite{KenigLinShen12} cannot be applied to operators with bounded measurable coefficients, since the correctors $\chi$ are not necessarily bounded. With the help of the Steklov smoothing operator, Suslina \cite{Suslina1301,Suslina1302} was able to establish the $O(\varepsilon)$ estimate (\ref{ineq.convergence-L2.Neumann}) in $L^2$ for a broader class of elliptic operators, which includes the standard elliptic systems in divergence form. Instead, $u_0+\varepsilon\chi^\varepsilon S_\varepsilon\left(\nabla\widetilde{u}_0\right)$ was used as the first order approximation of $u_\varepsilon$, and the following difference
\begin{equation}\label{intro.1}
v_\varepsilon=u_\varepsilon-u_0-\varepsilon\chi^\varepsilon S_\varepsilon\left(\nabla\widetilde{u}_0\right),
\end{equation}
was adopted, where $S_\varepsilon$ is the Steklov smoothing operator and $\widetilde{u}_0$ is an extension of $u_0$ to $\mathbb{R}^d$. %(also see \cite{Pastukhova06,PakhninSuslina13,ZhikovPastukhova05} and their references on the use of $S_\varepsilon$ in homogenization). %Besides of known convergence rate in $H^1$ and $L^2$ in $\mathbb{R}^d$ obtained in \cite{BirmanSuslina04,BirmanSuslina06}, the author introduce the boundary layer correction term $w_\varepsilon$ to approximate $v_\varepsilon$, which is the weak solution of $\mathcal{L}_\varepsilon(w_\varepsilon)=0$ in $\Omega$ and $\frac{\partial w_\varepsilon}{\partial \nu_\varepsilon}=\rho_\varepsilon$ on $\partial\Omega$. Here $\rho_\varepsilon$ is defined by
%\begin{equation}\label{e0.5}
%\langle \rho_\varepsilon,\varphi\rangle _{H^{-1/2}(\Omega)\times H^{1/2}(\Omega)}=\int_\Omega \left[A\nabla (\chi+P)\right]^\varepsilon S_\varepsilon \left(\nabla \widetilde{u}_0\right)\nabla (T\varphi)-F T\varphi,
%\end{equation}
%for any $\varphi\in H^{1/2}(\Omega;\mathbb{R}^d)$, where $T$ is the extension operator from $H^{1/2}$ to $H^1$. The estimate (\ref{ethm1.1a}) is obtained by the $O(\sqrt{\varepsilon})$ estimate of $\|w_\varepsilon\|_{H^1(\Omega)}$. Finally, consider the duality problem $\mathcal{L}_\varepsilon^*(\eta_\varepsilon)=\mathcal{L}_0^*(\eta_0)$ in $\Omega$ with $\left(\frac{\partial \eta_\varepsilon}{\partial \nu_\varepsilon}\right)^*=\left(\frac{\partial \eta_0}{\partial \nu_0}\right)^*=0$ on $\partial\Omega$, using an $O(\sqrt{\varepsilon})$ estimate in $H^1$ and duality argument, this allows one to prove an $O(\varepsilon)$ estimate of $\|w_\varepsilon\|_{L^2(\Omega)}$, and therefore the $L^2$ convergence (\ref{ethm1.2}).
For elliptic Neumann problems, the $O(\sqrt{\varepsilon})$ convergence (\ref{ineq.convergence-H1-u.Neumann}) in $H^1$ was obtained by the estimate of the boundary layer corrector term and the following sharp convergence rates for homogenization in the whole space 
$
\|u_\varepsilon-u_0\|_{L^2(\mathbb{R}^d)}\le C\varepsilon\|F\|_{L^2(\mathbb{R}^d)}.
$
The $O(\varepsilon)$ estimate (\ref{ineq.convergence-L2.Neumann}) in $L^2$ then can be deduced by applying the estimate (\ref{ineq.convergence-H1-u.Neumann}) to adjoint problems and a duality argument.

However, the case of Stokes systems certainly does not fit the standard framework of standard second-order elliptic systems in divergence form. As expected, in the study of Stokes or Navier-Stokes systems, the main difficulty is to deal with the pressure term $p_\varepsilon$. Because the conormal derivative includes the pressure term, it's not appropriate to use boundary layer corrector term as Suslina did for the elliptic Neumann problems \cite{Suslina1302} or as in \cite{Gu1501} for Dirichlet problems of Stokes systems. Instead, in this paper we use a more direct approach which helps us to avoid the convergence rates result for the whole space. The key intermediate step we use is the following,
$$
\left|\int_\Omega A^\varepsilon\nabla v_\varepsilon\cdot \nabla \varphi\right|\le C\|u_0\|_{H^2(\Omega)}\left[\|\text{div}(\varphi)\|_{L^2(\Omega)}+\varepsilon^{1/2}\|\nabla\varphi\|_{L^2(\Omega_{2\varepsilon})}+\varepsilon\|\nabla\varphi\|_{L^2(\Omega)}\right], 
$$
for any $\varphi\in H^1(\Omega;\mathbb{R}^d)$, and $\Omega_\varepsilon=\{x\in\Omega: \text{dist}(x,\partial \Omega)<\varepsilon\}$.
By choosing the suitable test function in the above key step, and provided the estimates of the correctors $(\chi, \pi)$ and their duals $(\Phi, q)$, we are able to establish the $O(\sqrt{\varepsilon})$ error estimates for the two-scale expansions of $(u_\varepsilon, p_\varepsilon)$ in $H^1\times  L^2$, which is stated in Theorem \ref{thm.convergence-H1.Neumann}. We emphasize that convergence rates of the pressure term in $L^2$ require an explicit computation of the pressure term corresponding to $v_\varepsilon$. At last, we apply the $O(\sqrt{\varepsilon})$ estimates to adjoint problems, and further obtain the $O(\varepsilon)$ estimate in $L^2$ through duality argument as well as the key intermediate step.

The theory of homogenization for operators with rapidly oscillating coefficients has been playing a vital part in describing the behavior of composite materials, which contain two or more finely mixed constituents. For Stokes systems with rapidly oscillating coefficients, one application would be studying the groundwater behavior in layered aquifer structures. Another application is related to the incompressible free fluid in porous media. If the ratio of the size of the porous to the period is $O(\varepsilon)$, and consider a  viscosity function which characterizes different viscosities in fluid and solid parts, then the homogenization theory in porous media and the derivation of Darcy's law may be regarded as one type of homogenization of Stokes systems with rapidly oscillating periodic coefficients.

The problem of convergence rates has been playing an essential role in quantitative homogenization. Most recent work on the problem of convergence rates in periodic homogenization may be found in \cite{JikovKozlovOleinik94,Griso04,Griso06,Pastukhova06,OnofreiVernescu07,KenigLinShen12,KenigLinShen14,PakhninSuslina13,Suslina1301,Suslina1302,Shen15,Gu1501,ShenZhuge15,GengShen16,ZhikovPastukhova05} and their references.

We now mention the potential applications of these results. Inspired by recent paper of Shen \cite{Shen15} on systems of linear elasticity, we expect to establish the boundary Lipschitz estimates in $C^{1,\alpha}$ domains for Neumann problems of Stokes systems with rapidly oscillating periodic coefficients, using convergence rates in $H^1$ and $L^2$ rather than the compactness method introduced by Avellaneda and Lin \cite{AL8701}. We may also use the result to investigate the $C^{\alpha}$, $W^{1,p}$, and $L^p$ estimates in $C^1$ domains with VMO or H\"older continuous coefficients.

The paper is organized as follows.
In Section 2 we state the homogenization theory of Stokes systems, estimates of dual correctors $(\Phi, q)$ and also the Steklov smoothing operator. In Section 3 we derive the key intermediate step by explicit computation of the system and conormal derivative that $v_\varepsilon$ satisfies, and further we prove the $O(\sqrt{\varepsilon})$ rate of $u_\varepsilon$ in $H^1$. 
In Section 4 we need a more explicit computation to prove the $O(\sqrt{\varepsilon})$ for the two-scale expansion of the pressure term $p_\varepsilon$. Finally, our main theorem Theorem \ref{thm.convergence-L2.Neumann} is proved in Section 5 by applying Theorem \ref{thm.convergence-H1.Neumann} to adjoint systems and also a duality argument.

Throughout this paper, we use $Y=[0,1)^d$ to denote the unit cube and define the $L^1$ average of $f$ over the set $E$ as
$$
\average_E f= \frac{1}{|E|}\int_E f.
$$
We will use $C$ to denote constants that may depend on $d$, $\mu$ or $\Omega$, but never on $\varepsilon$.

%%%%%%%%%%%%%%%%%%%%%%%%%%%%%%%%%%%%%%%%%%%%%%%%%%%%%%%%%%%%%%%%%%%%%%%%%%%%%%%%%%%%%%%%%%%%%%%%%%%%%%%%%%%%%%%%%%%%%%%%%%%%%%%%
\section{Preliminaries}
\setcounter{equation}{0}
%%%%%%%%%%%%%%%%%%%%%%%%%%%%%%%%%%%%%%%%%%%%%%%%%%%%%%%%%%%%%%%%

%%%%%%%%%%%%%%%%%%%%%%%%%%%%%%%%%%%%%%%%%%%%%%%%%%%%%%%%%%%%%%%
\subsection{Weak solution of Stokes systems}
We use this subsection to review the weak solutions of Stokes systems and the qualitative homogenization theorem for Neumann problems. Details may be found in \cite{Lions80,GuShen15,Gu16}.

Let $\Omega$ be a bounded Lipschitz domain in $\mathbb{R}^d$. Given $F\in H^{-1}(\Omega;\mathbb{R}^d)$ and $g\in L^2(\Omega)$, consider the following Stokes system 
\begin{equation}\label{def.Stokes}
\left\{
\begin{aligned}
\mathcal{L}_\varepsilon(u_\varepsilon)+\nabla p_\varepsilon &= F \\
\text{div}(u_\varepsilon) &=g\\
\end{aligned}
\right.
\qquad\text{in}\quad \Omega.
\end{equation}  
We define the bilinear form $a_\varepsilon(\cdot,\cdot)$ by
$$
a_\varepsilon(u,v)=\int_\Omega a_{ij}^{\alpha\beta}\left(\frac{x}{\varepsilon}\right)\frac{\partial u^\beta}{\partial x_j}\frac{\partial v^\alpha}{\partial x_i} dx, \qquad \forall\ u,v \in H^1(\Omega;\mathbb{R}^d),
$$
then we say that $(u_\varepsilon,p_\varepsilon)\in H^1(\Omega;\mathbb{R}^d)\times L^2(\Omega)$ is a weak solution of system (\ref{def.Stokes}) if 
$$
a_\varepsilon(u_\varepsilon,\varphi)-\int_\Omega p_\varepsilon \text{ div}(\varphi)=\langle F,\varphi\rangle, \qquad\forall\ \varphi \in C_0^1(\Omega;\mathbb{R}^d),
$$
and $\text{div}(u_\varepsilon)=g$ in $\Omega$ (in the sense of distribution).
\begin{theorem}\label{thm.weak-solution.Neumann}
Let $\Omega$ be a bounded Lipschitz domain in $\mathbb{R}^d$. Suppose $A(y)$ satisfies the ellipticity condition (\ref{cond.ellipticity}) and given $F\in H^{-1}(\Omega;\mathbb{R}^d)$, $g\in L^2(\Omega)$ and $f\in H^{-1/2}(\partial\Omega;\mathbb{R}^d)$ satisfying the compatibility condition (\ref{cond.compatibility.Neumann}). Then there exist a unique (up to constants) $u_\varepsilon \in H^1(\Omega;\mathbb{R}^d)$ and $p_\varepsilon \in L^2(\Omega)$ satisfying the following weak formulation
\begin{equation}\label{eq.weak-formulation}
a_\varepsilon(u_\varepsilon,\varphi)-\int_\Omega p_\varepsilon \text{\rm div}(\varphi)\,dx=\int_\Omega F\varphi \,dx+\int_{\partial \Omega} f\varphi \,dS, \quad \forall \varphi\in H^1(\Omega;\mathbb{R}^d),
\end{equation}
then we say $(u_\varepsilon,p_\varepsilon)$ is a weak solution of Stokes system (\ref{def.Stokes}) with Neumann boundary condition $\frac{\partial u_\varepsilon}{\partial \nu_\varepsilon}-p_\varepsilon \, n=f$ on $\partial\Omega$. Moreover, if $\int_\Omega u_\varepsilon\, dx=0$, then
\begin{equation}\label{ineq.energy.Neumann}
\|u_\varepsilon\|_{H^1(\Omega)}+\|p_\varepsilon-\average_\Omega p_\varepsilon\|_{L^2(\Omega)} \le C\left\{\|F\|_{H^{-1}(\Omega)}+\|g\|_{L^2(\Omega)}+\|f\|_{H^{-1/2}(\partial\Omega)}\right\},
\end{equation}
where $C$ depends only on $d$, $\mu$, and $\Omega$.
\end{theorem}
\begin{proof}
The proof is based on the Lax-Milgram Theorem. We skip the details here.
\end{proof}

\begin{remark}
Suppose $\Omega$ is $C^{1,1}$ and $A$ is a constant matrix, and provided that $F\in L^2(\Omega; \mathbb{R}^d)$, $g\in H^1(\Omega)$ and $f\in H^{1/2}(\partial\Omega; \mathbb{R}^d)$, then the weak solution $(u, p)$ given by Theorem \ref{thm.weak-solution.Neumann}, is in $H^2(\Omega; \mathbb{R}^d)\times H^1(\Omega)$.
Moreover, if $\int_\Omega u\, dx=0$, then
\begin{equation}\label{ineq.energy0.Neumann}
\|u\|_{H^2(\Omega)}+\|\nabla  p\|_{L^2(\Omega)} 
\le C\Big\{\|F\|_{L^2(\Omega)}+\|g\|_{H^1(\Omega)}+\|f\|_{H^{1/2}(\partial\Omega)}\Big\},
\end{equation}
where $C$ depends only on $d$, $\mu$, and $\Omega$ (see e.g. \cite{GiaquintaModica82}).
\end{remark}

\subsection{Correctors and Homogenization theorem}
For each $1\le j,\beta\le d$, we define the 1-periodic functions $(\chi_j^\beta,\pi_j^\beta)\in H^1_{\text{loc}}(\mathbb{R}^d;\mathbb{R}^d)\times L^2_{\text{loc}}(\mathbb{R}^d)$ as the correctors for the Stokes system (\ref{def.Stokes}), which satisfy the following cell problem
\begin{equation}\label{def.cell-problem}
\left\{
\begin{aligned}
\mathcal{L}_1(\chi_j^\beta+P_j^\beta)+\nabla \pi_j^\beta &= 0 \quad \text{in }\mathbb{R}^d, \\
\text{div}(\chi_j^\beta) &=0 \quad \text{in }\mathbb{R}^d,\\
\int_Y \pi_j^\beta=0, \quad \int_Y \chi_j^\beta &=0,\\
\end{aligned}
\right.
\end{equation}
where $P_j^\beta=P_j^\beta(y)=y_j e^\beta=y_j(0,\cdots,1,\cdots,0)$ with 1 in the $\beta^{\text{th}}$ position. Existence of such functions can be found in \cite{Gu1501}.
The homogenized system for the Stokes system (\ref{def.Stokes}) is given by
\begin{equation}\label{def.Stokes0}
\left\{
\begin{aligned}
\mathcal{L}_0(u_0)+\nabla p_0 &= F \\
\text{div}(u_0) &=g, \\
\end{aligned}
\right.
\end{equation}
where $\mathcal{L}_0=-\text{div}(\widehat{A}\nabla)$ is a second-order elliptic operator with constant coefficients $\widehat{A}=(\widehat{a}_{ij}^{\alpha\beta})$, with
\begin{equation*}
\widehat{a}_{ij}^{\alpha\beta}=\int_Y \left[a_{ij}^{\alpha\beta}(y)+a_{ik}^{\alpha\gamma}(y)\frac{\partial \chi^{\gamma \beta}_j(y)}{\partial y_k}\right]\, dy.
\end{equation*}
We should remark that $(\widehat{A})^*=\widehat{A^*}$, and the effective matrix $\widehat{A}$ is also elliptic. The following is a homogenization theorem for Neumann problem for the Stokes system.

\begin{theorem}\label{thm.homogenization.Neumann}
Suppose $A(y)$ satisfies the ellipticity (\ref{cond.ellipticity}) and periodicity  (\ref{cond.periodicity}) conditions. Let $\Omega$ be a bounded Lipschitz domain. Let $(u_\varepsilon,p_\varepsilon)\in H^1(\Omega;\mathbb{R}^d)\times L^2(\Omega)$ be a weak solution of Neumann problem (\ref{def.Stokes.Neumann}) in the sense of (\ref{eq.weak-formulation}), provided that $F\in H^{-1}(\Omega;\mathbb{R}^d)$, $g\in L^2(\Omega)$ and $f\in H^{-1/2}(\partial\Omega;\mathbb{R}^d)$ satisfying the compatibility condition (\ref{cond.compatibility.Neumann}). Assume that $\int_\Omega u_\varepsilon=\int_\Omega p_\varepsilon=0$, then as $\varepsilon \rightarrow 0$,
$$
\left\{
\begin{aligned}
u_\varepsilon &\rightarrow u_0 \quad\text{\rm strongly in }L^2(\Omega;\mathbb{R}^d),\\
u_\varepsilon &\rightharpoonup u_0 \quad\text{\rm weakly in }H^1(\Omega;\mathbb{R}^d),\\
p_\varepsilon &\rightharpoonup p_0 \quad\text{\rm weakly in }L^2(\Omega),\\
A(x/\varepsilon)\nabla u_\varepsilon &\rightharpoonup \widehat{A}\nabla u_0 \quad\text{\rm weakly in }L^2(\Omega;\mathbb{R}^{d\times d}).\\
\end{aligned}
\right.
$$
Moreover, $\int_\Omega u_0=\int_\Omega p_0=0$ and  $(u_0,p_0)$ is the weak solution of the homogenized problem (\ref{def.Stokes0.Neumann}), in the following weak sense
\begin{equation}\label{eq.weak-formulation0}
\int_{\Omega} \widehat{a}_{ij}^{\alpha\beta}\frac{\partial u_0^\beta}{\partial x_j}\frac{\partial \varphi^\alpha}{\partial x_i}\, dx-\int_\Omega p_0 \,\text{\rm div}(\varphi)\, dx=\int_\Omega F\varphi \, dx +\int_{\partial\Omega} f\varphi dS, \quad\forall \varphi\in H^1(\Omega;\mathbb{R}^d). 
\end{equation}
\end{theorem}

\begin{proof}
The proof of Theorem \ref{thm.homogenization.Neumann} uses Tartar's oscillating testing function method. Details can be found in \cite{Gu16}.
\end{proof}
%%%%%%%%%%%%%%%%%%%%%%%%%%%%%%%%%%%%%%%%%%%%%%%%%%%%%%%%%%%%%%%%
\subsection{Dual correctors of Stokes systems}
This subsection is used to introduce the dual correctors $(\Phi_{kij}^{\alpha\beta},q_{ij}^\beta)$ of Stokes systems and their properties. More details can be found in \cite{Gu1501}.

For $1\le i,j,\alpha,\beta \le d$, we let
\begin{equation}\label{def.b}
b_{ij}^{\alpha\beta}(y)=a_{ij}^{\alpha\beta}(y)+a_{ik}^{\alpha\gamma}(y)\frac{\partial}{\partial y_k}(\chi_{j}^{\gamma\beta})-\widehat{a}_{ij}^{\alpha\beta}.
\end{equation}
It is worth noting that $b_{ij}^{\alpha\beta}\in L^2(Y)$ is $1$-periodic with 
$\int_Y b_{ij}^{\alpha\beta}(y)dy=0$.%, and for each $1\le \alpha,\beta, j \le d$,
%\begin{equation}\label{eq.b-derivative}
%\aligned
%\frac{\partial}{\partial y_i}(b_{ij}^{\alpha\beta}(y))%&=\frac{\partial}{\partial y_i}(a_{ij}^{\alpha\beta}(y))+\frac{\partial}{\partial y_i}\left(a_{ik}^{\alpha\gamma}(y)\frac{\partial \chi_{j}^{\gamma\beta}}{\partial y_k}\right)\\
%&=\frac{\partial}{\partial y_i}(a_{ij}^{\alpha\beta}(y))-\frac{\partial}{\partial y_i}\left(a_{ik}^{\alpha\gamma}(y)\frac{\partial P_{j}^{\gamma\beta}}{\partial y_k}\right)+\frac{\partial}{\partial y_\alpha}(\pi_j^\beta) \\
%&=\frac{\partial}{\partial y_\alpha}(\pi_j^\beta) 
%\endaligned
%\end{equation}
\begin{lemma}\label{lem.dual-correctors}
There exist $\varPhi_{kij}^{\alpha\beta} \in H_{\text{per}}^1(Y)$ and $q_{ij}^\beta\in H_{\text{per}}^1(Y)$ such that
\begin{equation}\label{eq.dual-correctors}
b_{ij}^{\alpha\beta}=\frac{\partial}{\partial y_k}(\Phi_{kij}^{\alpha\beta})+\frac{\partial}{\partial y_\alpha}(q_{ij}^\beta) \quad \text{\rm and }\Phi_{kij}^{\alpha\beta}=-\Phi_{ikj}^{\alpha\beta}.
\end{equation}
Moreover,
\begin{equation}\label{ineq.dual-corrector-energy}
\|\Phi_{kij}^{\alpha\beta}\|_{L^2(Y)}+\|q_{ij}^{\beta}\|_{L^2(Y)} \le C,
\end{equation}
and with relation 
\begin{equation}\label{eq.q-pi}
\pi_j^\beta=\frac{\partial q_{ij}^\beta}{\partial y_i},
\end{equation}
where $C$ depends only on $d$ and $\mu$.
\end{lemma}

%%%%%%%%%%%%%%%%%%%%%%%%%%%%%%%%%%%%%%%%%%%%%%%%%%%%%%%%%%%%%%%%%%%%
\subsection{Steklov smoothing operator and boundary layer integrals}
Steklov smoothing operator will again play a vital part in this paper, we refer the readers to literature such as \cite{Pastukhova06,Suslina1301,Suslina1302,PakhninSuslina12,ZhikovPastukhova05} and their references for its detailed properties and applications.

We define the Steklov smoothing operator $S_\varepsilon$ in $L^2(\mathbb{R}^d;\mathbb{R}^d)$ by 
\begin{equation}\label{def.Steklov}
(S_\varepsilon u)(x)=\average_Y u(x-\varepsilon z)dz
\end{equation}
which satisfies the estimate $\|S_\varepsilon u\|_{L^2(\mathbb{R}^d)} \le \|u\|_{L^2(\mathbb{R}^d)}$. Obviously, $D^{\alpha}S_\varepsilon u=S_\varepsilon D^\alpha u$ for $u\in H^s(\mathbb{R}^d;\mathbb{R}^d)$ and any multi-index $\alpha$ such that $|\alpha|\le s$, and easily we can see that $\|S_\varepsilon u\|_{H^s(\mathbb{R}^d)} \le \|u\|_{H^s(\mathbb{R}^d)}$.

\begin{prop}\label{prop.steklov-difference}
For any $u\in H^1(\mathbb{R}^d)$ we have
$$
\|S_\varepsilon u-u\|_{L^2(\mathbb{R}^d)} \le C\varepsilon \|\nabla u\|_{L^2(\mathbb{R}^d)},
$$
where $C$ does not depend on $\varepsilon$.
\end{prop}
\begin{prop}\label{prop.steklov-multiple}
Let $f(x)$ be a $1$-periodic function in $\mathbb{R}^d$ such that $f\in L^2(Y)$. Then for any $u\in L^2(\mathbb{R}^d)$,
$$
\|f^\varepsilon S_\varepsilon u\|_{L^2(\mathbb{R}^d)} \le \|f\|_{L^2(Y)}\|u\|_{L^2(\mathbb{R}^d)}.
$$
\end{prop}
We define the $r$-neighborhood of the boundary $\partial\Omega$ by
\begin{equation}\label{def.r-neighborhood}
\aligned
(\partial\Omega)_r &=\{x\in\mathbb{R}^d: \text{dist}(x,\partial\Omega)\le r\},\\
\Omega_r & =\{x\in\Omega: \text{dist}(x,\partial\Omega)\le r\}.
\endaligned
\end{equation}
The following lemma gives us an estimate for integrals near the boundary, see \cite{Suslina1301,Suslina1302} for example. We will use it repeatedly in the following content.
\begin{lemma}\label{lem.near-boundary-integral}
Let $\Omega\subset \mathbb{R}^d$ be a bounded $C^1$ domain.
Then, for any function $u\in H^1(\Omega)$ and for any $0<r\le \textnormal{diam}(\Omega)$,
\begin{equation}\label{ineq.near-boundary-integral-inside}
\left(\int_{\Omega_r} |u|^2 dx\right)^{1/2} \le C \sqrt{r} \|u\|^{1/2}_{H^1(\Omega)}\|u\|^{1/2}_{L^2(\Omega)}.
\end{equation}
Moreover, for any 1-periodic function $f \in L^2(Y)$ and $u\in H^1(\mathbb{R}^d)$,
\begin{equation}\label{ineq.near-boundary-integral-outside}
\left(\int_{(\partial\Omega)_{2\varepsilon}}|f^\varepsilon|^2|S_\varepsilon u|^2 dx\right)^{1/2} \le C\sqrt{\varepsilon}\|f\|^{1/2}_{L^2(Y)}\|u\|^{1/2}_{H^1(\mathbb{R}^d)}\|u\|^{1/2}_{L^2(\mathbb{R}^d)},
\end{equation}
where $C$ depends only on $\Omega$.
\end{lemma}
%%%%%%%%%%%%%%%%%%%%%%%%%%%%%%%%%%%%%%%%%%%%%%%%%%%%%%%%%%%%

\section{Convergence rates for $u_\varepsilon$ in $H^1$}
\setcounter{equation}{0}

From now on we will assume that
$\Omega$ is a bounded $C^{1,1}$ domain, and $F\in L^2(\Omega; \mathbb{R}^d)$, $g\in H^1(\Omega)$, and $f\in H^{1/2}(\partial\Omega; \mathbb{R}^d)$ satisfy the compatibility condition (\ref{cond.compatibility.Neumann}). We further assume that $\int_\Omega u_\varepsilon=\int_\Omega u_0=0$, and also $\int_\Omega p_\varepsilon=\int_\Omega p_0=0$. We define a linear continuous extension operator $E_\Omega: H^2(\Omega;\mathbb{R}^d) \rightarrow H^2(\mathbb{R}^d;\mathbb{R}^d)$
and for simplicity we denote 
\begin{equation}\label{def.extension-operator}
\widetilde{u}_0=E_\Omega u_0,
\end{equation} 
so that $\widetilde{u}_0=u_0$ in $\Omega$ and
\begin{equation}\label{ineq.extension}
\|\widetilde{u}_0\|_{H^2(\mathbb{R}^d)} \le C\|u_0\|_{H^2(\Omega)},
\end{equation}
where $C$ depends on $\Omega$.  

We define $u_0+\varepsilon \chi^\varepsilon S_\varepsilon \left(\nabla \widetilde{u}_0\right)$ as a first order approximation of $u_\varepsilon$, and let
\begin{equation}\label{def.v}
v_\varepsilon=u_\varepsilon-u_0-\varepsilon \chi^\varepsilon S_\varepsilon \left(\nabla \widetilde{u}_0\right).
\end{equation}
To prove (\ref{ineq.convergence-H1-u.Neumann}), we need to show $\|v_\varepsilon\|_{H^1(\Omega)}\le C\sqrt{\varepsilon}\|u_0\|_{H^2(\Omega)}$. 
\begin{lemma}\label{lem.v-bilinear-form}
Let $\Omega$ be a bounded $C^{1,1}$ domain. Suppose $A$ satisfies ellipticity condition (\ref{cond.ellipticity}) and periodicity condition (\ref{cond.periodicity}). Given $F\in L^2(\Omega;\mathbb{R}^d)$ and $f\in H^{1/2}(\partial\Omega;\mathbb{R}^d)$ satisfying the compatibility condition (\ref{cond.compatibility.Neumann}),  for $g\in H^1(\Omega)$, $(u_\varepsilon,p_\varepsilon)$, $(u_0,p_0)$ are weak solutions of Neumann problems (\ref{def.Stokes.Neumann}) and (\ref{def.Stokes0.Neumann}), respectively. If $\int_\Omega u_\varepsilon=\int_\Omega u_0=0$ and $\int_\Omega p_\varepsilon=\int_\Omega p_0=0$, and $v_\varepsilon$ is defined as in (\ref{def.v}), then $v_\varepsilon$ satisfies the following weak formulation
\begin{equation}\label{eq.lem.v-bilinear-form}
\begin{aligned}
a_\varepsilon(v_\varepsilon,\varphi)&=\int_\Omega [p_\varepsilon-p_0]\,\text{\rm div}(\varphi)-\int_\Omega [\widehat{A}-A^\varepsilon][S_\varepsilon\left(\nabla \widetilde{u}_0\right)-\nabla u_0]\cdot \nabla \varphi\\
&\quad -\varepsilon\int_\Omega A^\varepsilon \chi^\varepsilon S_\varepsilon \left(\nabla^2 \widetilde{u}_0\right)\cdot \nabla \varphi -\int_\Omega B^\varepsilon S_\varepsilon\left(\nabla \widetilde{u}_0\right)\cdot \nabla \varphi, \quad\forall \varphi\in H^1(\Omega;\mathbb{R}^d).
\end{aligned}
\end{equation}
\end{lemma}

\begin{proof}
Since $(u_\varepsilon,p_\varepsilon)$, $(u_0,p_0)$ satisfy the weak formulations (\ref{eq.weak-formulation}) and (\ref{eq.weak-formulation0}), respectively, therefore for any $\varphi\in H^1(\Omega;\mathbb{R}^d)$,
\begin{equation}\label{pf1.v-weak-formulation}
a_\varepsilon(u_\varepsilon,\varphi)-\int_\Omega p_\varepsilon \text{ div}(\varphi) =\int_{\Omega} \widehat{a}_{ij}^{\alpha\beta}\frac{\partial u_0^\beta}{\partial x_j}\frac{\partial \varphi^\alpha}{\partial x_i}-\int_\Omega p_0 \,\text{\rm div}(\varphi).
\end{equation}
By (\ref{pf1.v-weak-formulation}) and the definition of $v_\varepsilon$, we have
\begin{equation}\label{pf2.v-weak-formulation}
\begin{aligned}
a_\varepsilon(v_\varepsilon,\varphi)&=a_\varepsilon(u_\varepsilon,\varphi)-a_\varepsilon(u_0,\varphi)-a_\varepsilon(\varepsilon\chi^\varepsilon S_\varepsilon(\nabla \widetilde{u}_0),\varphi)\\
&=\int_\Omega (p_\varepsilon-p_0)\text{ div}(\varphi)+\int_\Omega \big[\widetilde{a}_{ij}^{\alpha\beta}-a_{ij}^{\alpha\beta}(x/\varepsilon)\big]\frac{\partial u_0^\beta}{\partial x_j}\frac{\partial \varphi^\alpha}{\partial x_i}-a_\varepsilon(\varepsilon\chi^\varepsilon S_\varepsilon(\nabla \widetilde{u}_0),\varphi).
\end{aligned}
\end{equation}
Since 
\begin{equation}
\begin{aligned}
a_\varepsilon(\varepsilon\chi^\varepsilon S_\varepsilon(\nabla \widetilde{u}_0),\varphi)=\int_\Omega a_{ik}^{\alpha\gamma}(x/\varepsilon)\frac{\partial}{\partial x_k}&\left(\varepsilon\chi_j^{\gamma\beta}(x/\varepsilon)\right)S_\varepsilon\frac{\partial \widetilde{u}_0^\beta}{\partial x_j}\frac{\partial \varphi^\alpha}{\partial x_i}\\
&+\varepsilon\int_\Omega a_{ik}^{\alpha\gamma}(x/\varepsilon)\chi_j^{\gamma\beta}(x/\varepsilon)S_\varepsilon\frac{\partial^2 \widetilde{u}_0^\beta}{\partial x_k \partial x_j}\frac{\partial \varphi^\alpha}{\partial x_i},
\end{aligned}
\end{equation}
then by using the definition of $b$, we can obtain that
\begin{equation}\label{pf3.v-weak-formulation}
\begin{aligned}
a_\varepsilon(v_\varepsilon,\varphi)&=\int_\Omega (p_\varepsilon-p_0)\text{ div}(\varphi)-\int_\Omega \Big[\widehat{a}_{ij}^{\alpha\beta}-a_{ij}^{\alpha\beta}(x/\varepsilon)\Big]\bigg[S_\varepsilon\frac{\partial \widetilde{u}_0^\beta}{\partial x_j}-\frac{\partial u_0^\beta}{\partial x_j}\bigg]\frac{\partial \varphi^\alpha}{\partial x_i}\\
&\quad -\varepsilon\int_\Omega a_{ik}^{\alpha\gamma}(x/\varepsilon)\chi_j^{\gamma\beta}(x/\varepsilon)S_\varepsilon\frac{\partial^2 \widetilde{u}_0^\beta}{\partial x_k \partial x_j}\frac{\partial \varphi^\alpha}{\partial x_i}-\int_\Omega b_{ij}^{\alpha\beta}(x/\varepsilon)S_\varepsilon\frac{\partial \widetilde{u}_0^\beta}{\partial x_j}\frac{\partial \varphi^\alpha}{\partial x_i}.\\
\end{aligned}
\end{equation}

\end{proof}
We choose a cut-off function $\theta_\varepsilon(x)$ in $\mathbb{R}^d$ satisfying  the following conditions,
\begin{equation}\label{cutoff1}
\begin{aligned}
& \theta_\varepsilon\in C_0^\infty(\mathbb{R}^d), \quad \text{supp} (\theta_\varepsilon) \subset (\partial \Omega)_\varepsilon,\quad 0\le \theta_\varepsilon(x)\le 1,\\
& \theta_\varepsilon|_{\partial\Omega}=1, \quad |\nabla \theta_\varepsilon|\le \kappa/\varepsilon.
\end{aligned}
\end{equation}
%and
%\begin{equation}\label{cutoff2}
%\begin{aligned}
%& \widetilde{\theta}_\varepsilon\in C_0^\infty(\mathbb{R}^d), \quad \text{supp} (\widetilde{\theta}_\varepsilon) \subset (\partial \Omega)_{2\varepsilon},\quad 0\le \widetilde{\theta}_\varepsilon(x)\le 1,\\
%& \widetilde{\theta}_\varepsilon(x)=1 \text{ for } x\in (\partial\Omega)_\varepsilon, 
%\quad |\nabla \widetilde{\theta}_\varepsilon |\le \widetilde{\kappa}/\varepsilon.
%\end{aligned}
%\end{equation}
The following lemma is the key intermediate step to prove our convergence results.
\begin{lemma}\label{lem.key-step}
Let $v_\varepsilon$ be defined as in (\ref{def.v}). Then, for any $\varphi\in H^1(\Omega;\mathbb{R}^d)$ we have
\begin{equation}\label{ineq.key-step}
|a_\varepsilon(v_\varepsilon,\varphi)|\le C\|u_0\|_{H^2(\Omega)}\left[\|\text{\rm div}(\varphi)\|_{L^2(\Omega)}+\varepsilon^{1/2}\|\nabla\varphi\|_{L^2(\Omega_{2\varepsilon})}+\varepsilon\|\nabla\varphi\|_{L^2(\Omega)}\right],
\end{equation}
where the constant $C$ depends only on $\mu$, $d$, and $\Omega$.
\end{lemma}
\begin{proof}
For any $\varphi \in H^1(\Omega;\mathbb{R}^d)$, by Lemma \ref{lem.v-bilinear-form}, we have proved that $v_\varepsilon$ satisfies the weak formulation (\ref{pf3.v-weak-formulation}). By Lemma \ref{lem.dual-correctors}, we may rewrite the last term in RHS of (\ref{pf3.v-weak-formulation}) as
\begin{equation}\label{pf2.key-step}
\begin{aligned}
%-\int_\Omega b_{ij}^{\alpha\beta}(x/\varepsilon)S_\varepsilon\frac{\partial \widetilde{u}_0^\beta}{\partial x_j}\frac{\partial \varphi^\alpha}{\partial x_i}
-\int_\Omega \frac{\partial}{\partial x_k}\left(\varepsilon\Phi_{kij}^{\alpha\beta}(x/\varepsilon)\right)S_\varepsilon\frac{\partial \widetilde{u}_0^\beta}{\partial x_j}\frac{\partial \varphi^\alpha}{\partial x_i}-\int_\Omega \frac{\partial}{\partial x_\alpha}\left(\varepsilon q_{ij}^{\beta}(x/\varepsilon)\right)S_\varepsilon\frac{\partial \widetilde{u}_0^\beta}{\partial x_j}\frac{\partial \varphi^\alpha}{\partial x_i}=R_1+R_2.\\
\end{aligned}
\end{equation}
For the first integral $R_1$, we see that
\begin{equation}\label{pf3.key-step}
\begin{aligned}
R_1=-\varepsilon\int_\Omega \frac{\partial}{\partial x_k}\left(\Phi_{kij}^{\alpha\beta}(x/\varepsilon)S_\varepsilon\frac{\partial \widetilde{u}_0^\beta}{\partial x_j}\right)\frac{\partial \varphi^\alpha}{\partial x_i}+\varepsilon\int_\Omega \Phi_{kij}^{\alpha\beta}(x/\varepsilon)S_\varepsilon\frac{\partial^2 \widetilde{u}_0^\beta}{\partial x_k \partial x_j}\frac{\partial \varphi^\alpha}{\partial x_i}=R_{1a}+R_{1b}.
\end{aligned}
\end{equation}
For the term $R_{1a}$, recall that $\Phi$ is anti-symmetric and $\theta_\varepsilon\equiv 1$ on $\partial\Omega$, therefore
\begin{equation}\label{pf4.key-step}
\begin{aligned}
&R_{1a}=-\varepsilon\int_\Omega \frac{\partial}{\partial x_k}\left(\big[\theta_\varepsilon+1-\theta_\varepsilon\big]\Phi_{kij}^{\alpha\beta}(x/\varepsilon)S_\varepsilon\frac{\partial \widetilde{u}_0^\beta}{\partial x_j}\right)\frac{\partial \varphi^\alpha}{\partial x_i}\\
&=-\varepsilon\int_\Omega \frac{\partial}{\partial x_k}\left(\theta_\varepsilon\Phi_{kij}^{\alpha\beta}(x/\varepsilon)S_\varepsilon\frac{\partial \widetilde{u}_0^\beta}{\partial x_j}\right)\frac{\partial \varphi^\alpha}{\partial x_i}+\varepsilon\int_\Omega \frac{\partial^2}{\partial x_k\partial x_i}\left((1-\theta_\varepsilon)\Phi_{kij}^{\alpha\beta}(x/\varepsilon)S_\varepsilon\frac{\partial \widetilde{u}_0^\beta}{\partial x_j}\right)\varphi^\alpha\\
&=-\varepsilon\int_{\Omega} \frac{\partial}{\partial x_k}\left(\theta_\varepsilon\Phi_{kij}^{\alpha\beta}(x/\varepsilon)S_\varepsilon\frac{\partial \widetilde{u}_0^\beta}{\partial x_j}\right)\frac{\partial \varphi^\alpha}{\partial x_i}.
\end{aligned}
\end{equation}
Similarly for $R_2$, we write it as 
\begin{equation}\label{pf5.key-step}
\begin{aligned}
R_2=-\varepsilon\int_\Omega \frac{\partial}{\partial x_\alpha}\left(q_{ij}^{\beta}(x/\varepsilon)S_\varepsilon\frac{\partial \widetilde{u}_0^\beta}{\partial x_j}\right)\frac{\partial \varphi^\alpha}{\partial x_i}+\varepsilon\int_\Omega q_{ij}^{\beta}(x/\varepsilon)S_\varepsilon\frac{\partial^2 \widetilde{u}_0^\beta}{\partial x_k \partial x_j}\frac{\partial \varphi^\alpha}{\partial x_i}=R_{2a}+R_{2b},
\end{aligned}
\end{equation}
and analogously for the term $R_{2a}$, since $\theta_\varepsilon\equiv 1$ on $\partial\Omega$,
\begin{equation}\label{pf6.key-step}
\begin{aligned}
&R_{2a}=-\varepsilon\int_\Omega \frac{\partial}{\partial x_\alpha}\left(\big[\theta_\varepsilon+1-\theta_\varepsilon\big]q_{ij}^{\beta}(x/\varepsilon)S_\varepsilon\frac{\partial \widetilde{u}_0^\beta}{\partial x_j}\right)\frac{\partial \varphi^\alpha}{\partial x_i}\\
&=-\varepsilon\int_\Omega \frac{\partial}{\partial x_\alpha}\left(\theta_\varepsilon q_{ij}^{\beta}(x/\varepsilon)S_\varepsilon\frac{\partial \widetilde{u}_0^\beta}{\partial x_j}\right)\frac{\partial \varphi^\alpha}{\partial x_i}+\varepsilon\int_\Omega \frac{\partial^2}{\partial x_\alpha\partial x_i}\left((1-\theta_\varepsilon) q_{ij}^{\beta}(x/\varepsilon)S_\varepsilon\frac{\partial \widetilde{u}_0^\beta}{\partial x_j}\right)\varphi^\alpha\\
&=-\varepsilon\int_{\Omega} \frac{\partial}{\partial x_\alpha}\left(\theta_\varepsilon q_{ij}^{\beta}(x/\varepsilon)S_\varepsilon\frac{\partial \widetilde{u}_0^\beta}{\partial x_j}\right)\frac{\partial \varphi^\alpha}{\partial x_i}-\varepsilon\int_\Omega \frac{\partial}{\partial x_i}\left((1-\theta_\varepsilon) q_{ij}^{\beta}(x/\varepsilon)S_\varepsilon\frac{\partial \widetilde{u}_0^\beta}{\partial x_j}\right)\text{div}(\varphi) .
\end{aligned}
\end{equation}
By (\ref{eq.q-pi}), it is more precise that
\begin{equation}\label{pf7.key-step}
\begin{aligned}
&R_{2a}=\int_{\Omega} \bigg[\Big(\varepsilon\frac{\partial \theta_\varepsilon}{\partial x_i}q_{ij}^\beta(x/\varepsilon)-(1-\theta_\varepsilon)\pi_j^\beta(x/\varepsilon)\Big)S_\varepsilon\frac{\partial \widetilde{u}_0^\beta}{\partial x_j}-\varepsilon (1-\theta_\varepsilon)q_{ij}^\beta(x/\varepsilon)S_\varepsilon\frac{\partial^2 \widetilde{u}_0^\beta}{\partial x_i\partial x_j}\bigg]\text{div}(\varphi)\\
&\qquad -\varepsilon\int_{\Omega} \frac{\partial}{\partial x_\alpha}\left(\theta_\varepsilon q_{ij}^{\beta}(x/\varepsilon)S_\varepsilon\frac{\partial \widetilde{u}_0^\beta}{\partial x_j}\right)\frac{\partial \varphi^\alpha}{\partial x_i}.
\end{aligned}
\end{equation}
Therefore, we have updated (\ref{pf3.v-weak-formulation}) as
\begin{equation}\label{eq.I1-plus-I2-plus-I3}
\begin{aligned}
a_\varepsilon(v_\varepsilon,\varphi)
%&= \int_\Omega \Big[p_\varepsilon-p_0-\big[(1-\theta_\varepsilon)\pi^\varepsilon-\varepsilon\nabla \theta_\varepsilon q^\varepsilon] S_\varepsilon(\nabla\widetilde{u}_0)-\varepsilon (1-\theta_\varepsilon)q^\varepsilon S_\varepsilon(\nabla^2 \widetilde{u}_0)\Big]\text{div} (\varphi)\\
%& +\varepsilon\int_\Omega \Big[q^\varepsilon S_\varepsilon(\nabla^2\widetilde{u}_0)\Big]\cdot \nabla \varphi+\varepsilon\int_\Omega \Big[(\Phi^\varepsilon-A^\varepsilon\chi^\varepsilon) S_\varepsilon (\nabla^2\widetilde{u}_0)\Big]\cdot \nabla \varphi\\
%& -\int_\Omega \Big[(\widehat{A}-A^\varepsilon)\big(S_\varepsilon(\nabla \widetilde{u}_0)-\nabla u_0\big)\Big]\cdot\nabla\varphi-\int_{\Omega} \Big[\nabla\big(\varepsilon \theta_\varepsilon \Phi^\varepsilon S_\varepsilon (\nabla\widetilde{u}_0)\big)\Big]\cdot \nabla \varphi\\
%&-\int_{\Omega} \Big[\nabla\big(\varepsilon \theta_\varepsilon q^\varepsilon S_\varepsilon (\nabla\widetilde{u}_0)\big)\Big]\cdot \nabla \varphi\\
= I_1[\varphi]+I_2[\varphi]+I_3[\varphi],\\
\end{aligned}
\end{equation}
where $I_1$, $I_2$, and $I_3$ are defined as the following,
\begin{equation}\label{def.I1-I2-I3}
\begin{aligned}
& I_1[\varphi]=\int_\Omega \Big[p_\varepsilon-p_0-\big[(1-\theta_\varepsilon)\pi^\varepsilon-\varepsilon\nabla \theta_\varepsilon q^\varepsilon] S_\varepsilon(\nabla\widetilde{u}_0)-\varepsilon (1-\theta_\varepsilon)q^\varepsilon S_\varepsilon(\nabla^2 \widetilde{u}_0)\Big]\text{div} (\varphi);\\
& I_2[\varphi]=\varepsilon\int_\Omega \Big[q^\varepsilon S_\varepsilon(\nabla^2\widetilde{u}_0)\Big]\cdot \nabla \varphi+\varepsilon\int_\Omega \Big[(\Phi^\varepsilon-A^\varepsilon\chi^\varepsilon) S_\varepsilon (\nabla^2\widetilde{u}_0)\Big]\cdot \nabla \varphi\\
&\qquad\quad-\int_\Omega \Big[(\widehat{A}-A^\varepsilon)\big(S_\varepsilon(\nabla \widetilde{u}_0)-\nabla u_0\big)\Big]\cdot\nabla\varphi;\\
& I_3[\varphi]= -\int_{\Omega} \Big[\nabla\big(\varepsilon \theta_\varepsilon \Phi^\varepsilon S_\varepsilon (\nabla\widetilde{u}_0)\big)\Big]\cdot \nabla \varphi-\int_{\Omega} \Big[\nabla\big(\varepsilon \theta_\varepsilon q^\varepsilon S_\varepsilon (\nabla\widetilde{u}_0)\big)\Big]\cdot \nabla \varphi.
\end{aligned}
\end{equation}
To estimate $I_1$, since $(v_\varepsilon, p_\varepsilon-p_0)$ satisfies the Stokes system (\ref{eq.lem.v-bilinear-form}) with data only involves with the gradient of $u_0$, therefore by energy estimate (\ref{ineq.energy.Neumann}), Propositions \ref{prop.steklov-difference}-\ref{prop.steklov-multiple} and the assumption that $\int_\Omega p_\varepsilon=\int_\Omega p_0=0$, we obtain that
\begin{equation}\label{I1.temp1}
\|p_\varepsilon-p_0\|_{L^2(\Omega)}\le C\|u_0\|_{H^2(\Omega)}.
\end{equation}
Since $\pi\in L^2(Y)$, $q\in H^1(Y;\mathbb{R}^d)$ and $|\nabla \theta_\varepsilon|\le \kappa/\varepsilon$, by Proposition \ref{prop.steklov-multiple}, H\"older's inequality and (\ref{I1.temp1}), we have
\begin{equation}\label{I1.key-step}
|I_1[\varphi]|\le C\|u_0\|_{H^2(\Omega)}\|\text{div}(\varphi)\|_{L^2(\Omega)},
%&\le \|p_\varepsilon-p_0-\big[(1-\theta_\varepsilon)\pi^\varepsilon+\varepsilon\nabla \theta_\varepsilon q^\varepsilon] S_\varepsilon(\nabla\widetilde{u}_0)-\varepsilon (1-\theta_\varepsilon)q^\varepsilon S_\varepsilon(\nabla^2 \widetilde{u}_0)\|_{L^2(\Omega)}\|\text{div}(\varphi)\|_{L^2(\Omega)}\\
\end{equation}
where $C$ depends only on $d$, $\mu$, and $\Omega$.

By Propositions \ref{prop.steklov-difference}-\ref{prop.steklov-multiple}, (\ref{ineq.extension}), and H\"older's inequality, we can obtain that
\begin{equation}\label{I2.key-step}
\begin{aligned}
|I_2[\varphi]|&\le C\varepsilon\left(\big[\|\chi\|_{L^2(Y)}+\|\Phi\|_{L^2(Y)}+\|q\|_{L^2(Y)}+1\big]\|\widetilde{u}_0\|_{H^2(\mathbb{R}^{d})}\right)\|\nabla \varphi\|_{L^2(\Omega)}\\
&\le C\varepsilon\|u_0\|_{H^2(\Omega)}\|\nabla \varphi\|_{L^2(\Omega)}.
\end{aligned}
\end{equation}
For $I_3$, since $\text{supp } \theta_\varepsilon\subset (\partial\Omega)_\varepsilon$ and again by H\"older's inequality, we have
\begin{equation}\label{pf8.key-step}
%\begin{aligned}
|I_3[\varphi]|\le C\left(\|\varepsilon\theta_\varepsilon q^\varepsilon S_\varepsilon (\nabla\widetilde{u}_0)\|_{H^{1}(\Omega)}+\|\varepsilon\theta_\varepsilon \Phi^\varepsilon S_\varepsilon (\nabla\widetilde{u}_0)\|_{H^{1}(\Omega)}\right)\|\nabla \varphi\|_{L^2(\Omega_{2\varepsilon})}.\\
%\end{aligned}
\end{equation}
For the first term in the parentheses of (\ref{pf8.key-step}), 
\begin{equation}\label{pf9.key-step}
\begin{aligned}
\varepsilon\|\theta_\varepsilon q^\varepsilon S_\varepsilon (\nabla\widetilde{u}_0)\|_{H^{1}(\Omega)}
&\le C \varepsilon\Big\{ \|q^\varepsilon S_\varepsilon (\nabla \widetilde{u}_0) \|_{L^2(\Omega)}+\|(\nabla \theta_\varepsilon) q^\varepsilon S_\varepsilon (\nabla \widetilde{u}_0)\|_{L^2(\Omega)}\\
&\qquad +\varepsilon^{-1}\|\theta_\varepsilon (\nabla q)^\varepsilon S_\varepsilon (\nabla \widetilde{u}_0) \|_{L^2(\Omega)}+\|q^\varepsilon S_\varepsilon (\nabla^2 \widetilde{u}_0)\|_{L^2(\Omega)} \Big\}\\
&\le C \varepsilon\Big\{ \| \widetilde{u}_0\|_{H^2(\mathbb{R}^d)}
+\varepsilon^{-1} \|q^\varepsilon S_\varepsilon (\nabla \widetilde{u}_0)\|_{L^2((\partial\Omega)_\varepsilon)}\\
&\qquad+\varepsilon^{-1} \|  (\nabla q)^\varepsilon S_\varepsilon (\nabla \widetilde{u}_0) \|_{L^2((\partial\Omega)_\varepsilon)}\Big\}\\
&\le C \varepsilon^{1/2}  \| \widetilde{u}_0\|_{H^2(\mathbb{R}^d)},
\end{aligned}
\end{equation}
where we have used Proposition \ref{prop.steklov-multiple} for the second inequality and Lemma \ref{lem.near-boundary-integral} for the last. Using the same manner, we can show that
\begin{equation}\label{pf10.key-step}
\varepsilon\|\theta_\varepsilon \Phi^\varepsilon S_\varepsilon (\nabla\widetilde{u}_0)\|_{H^{1}(\Omega)}\le C \varepsilon^{1/2}  \| \widetilde{u}_0\|_{H^2(\mathbb{R}^d)}.
\end{equation}
Therefore, by (\ref{ineq.extension}), we have proved that
\begin{equation}\label{I3.key-step}
|I_3[\varphi]|\le C\varepsilon^{1/2}\|u_0\|_{H^2(\Omega)}\|\nabla \varphi\|_{L^2(\Omega_{2\varepsilon})}
\end{equation}
Hence, by combining (\ref{I1.key-step}), (\ref{I2.key-step}) and (\ref{I3.key-step}), we know that
$$
|a_\varepsilon(v_\varepsilon,\varphi)|\le C\|u_0\|_{H^2(\Omega)}\left[\|\text{div}(\varphi)\|_{L^2(\Omega)}+\varepsilon^{1/2}\|\nabla\varphi\|_{L^2(\Omega_{2\varepsilon})}+\varepsilon\|\nabla\varphi\|_{L^2(\Omega)}\right],
$$
where $C$ depends only on $d$, $\mu$, and $\Omega$.
\end{proof}
\begin{proof}[\textbf{Proof of estimate (\ref{ineq.convergence-H1-u.Neumann})}]
We will now prove (\ref{ineq.convergence-H1-u.Neumann}) by energy estimates. Since $\text{div}(u_\varepsilon)=\text{div}(u_0)$, and recall that $\text{div}(\chi)=0$, hence
$$
\text{\rm div}(v_\varepsilon)=-\varepsilon\chi_j^{\alpha\beta}(x/\varepsilon) S_\varepsilon \frac{\partial^2 \widetilde{u}_0^\beta}{\partial x_\alpha\partial x_j}.
$$ 
By Proposition \ref{prop.steklov-multiple}, we can see that
\begin{equation}\label{pf1.convergence-H1-u.Neumann}
\|\text{div}(v_\varepsilon)\|_{L^2(\Omega)} \le C\varepsilon\|u_0\|_{H^2(\Omega)}
\end{equation}
If we choose $\varphi=v_\varepsilon$ itself in Lemma \ref{lem.key-step}, therefore by ellipticity condition (\ref{cond.ellipticity}), (\ref{ineq.key-step}) and (\ref{pf1.convergence-H1-u.Neumann}), we see that 
\begin{equation}\label{pf3.convergence-H1-u.Neumann}
\begin{aligned}
\|\nabla v_\varepsilon\|^2_{L^2(\Omega)} \le C\|u_0\|_{H^2(\Omega)}\left[\varepsilon^{1/2}\|\nabla v_\varepsilon\|_{L^2(\Omega_{2\varepsilon})}+\varepsilon\|\nabla v_\varepsilon\|_{L^2(\Omega)}\right]+C\varepsilon\|u_0\|^2_{H^2(\Omega)}.
\end{aligned}
\end{equation}
By Poincar\'{e}'s inequality, and since $\int_\Omega u_\varepsilon=\int_\Omega u_0=0$, we have
\begin{equation}\label{pf4.convergence-H1-u.Neumann}
\begin{aligned}
\|v_\varepsilon\|_{L^2(\Omega)}&\le \|v_\varepsilon-\average_\Omega v_\varepsilon \|_{L^2(\Omega)}+\|\average_\Omega \left[\varepsilon\chi^\varepsilon S_\varepsilon(\nabla\widetilde{u}_0)\right]\|_{L^2(\Omega)}\\
&\le C\|\nabla v_\varepsilon\|_{L^2(\Omega)}+C\varepsilon\|u_0\|_{H^2(\Omega)},
\end{aligned}
\end{equation}
where we have also used Proposition \ref{prop.steklov-multiple} and (\ref{ineq.extension}) for the last inequality. Finally, by (\ref{pf3.convergence-H1-u.Neumann}), (\ref{pf4.convergence-H1-u.Neumann}) and Cauchy-Schwarz inequality, we proved the desired result
\begin{equation}\label{pf2.convergence-H1-u.Neumann}
\|v_\varepsilon\|_{H^1(\Omega)}\le C\sqrt{\varepsilon}\|u_0\|_{H^2(\Omega)},
\end{equation}
where $C$ depends only on $d$, $\mu$, and $\Omega$.
\end{proof}
\begin{remark}
This key intermediate step (\ref{ineq.key-step}) has a nature advantage in the Neumann boundary value problems, because the boundary integral of the weak formulation no longer contains the pressure term. But as we mentioned earlier, Neumann problems are always more complicated than Dirichlet problems. Hence, with a slight modification, this method is also suitable for the Dirichlet convergence rate problem as we did in \cite{Gu1501}. We need to choose a Dirichlet boundary corrector $(w_\varepsilon,\tau_\varepsilon)$, which satisfies $\mathcal{L}_\varepsilon(w_\varepsilon)+\nabla \tau_\varepsilon=0$, $\text{\rm div}(w_\varepsilon)=\varepsilon\text{\rm div}(\chi^\varepsilon S_\varepsilon \nabla \widetilde{u}_0)$ and $w_\varepsilon|_{\partial\Omega}=\varepsilon\chi^\varepsilon S_\varepsilon \nabla \widetilde{u}_0$. Let $z_\varepsilon=u_\varepsilon-u_0-\varepsilon\chi^\varepsilon S_\varepsilon \nabla \widetilde{u}_0+w_\varepsilon$, since $z_\varepsilon$ belongs to the Hilbert space $V=\{u\in H^1_0(\Omega;\mathbb{R}^d): \text{\rm div}(u)=0\}$, it is not hard to see that 
$|a_\varepsilon(z_\varepsilon,\varphi)|\le C\varepsilon\|u_0\|_{H^2(\Omega)}\|\nabla \varphi\|_{L^2(\Omega)}$, for any $\varphi\in V$. Therefore the $O(\sqrt{\varepsilon})$ rate can be obtained by choosing $\varphi=z_\varepsilon$ and the fact that $\|w_\varepsilon\|_{H^1(\Omega)}\le C\sqrt{\varepsilon}\|u_0\|_{H^2(\Omega)}$. This gives us another aspect of seeing the importance of the choice of boundary corrector $w_\varepsilon$ in \cite{Gu1501}.
\end{remark}
%%%%%%%%%%%%%%%%%%%%%%%%%%%%%%%%%%%%%%%%%%%%%%%%%%%%%%%%%%%%%%%%%%%%%%%%%%%%%%%%%%%%%%%%%%%%%%%%%%%%%%%%%%%%%%%%%%%%%%%%%%%%%%%%
\section{Convergence rates of $p_\varepsilon$ in $L^2$}
\setcounter{equation}{0}
It is well known that if $(u_\varepsilon,p_\varepsilon)\in H^1(\Omega;\mathbb{R}^d)\times L^2(\Omega)$ is a weak solution of any Stokes system (\ref{def.Stokes}), then
\begin{equation}\label{ineq.inherent-estimate}
\|p_\varepsilon-\average_\Omega p_\varepsilon\|_{L^2(\Omega)}\le \|\nabla p_\varepsilon\|_{H^{-1}(\Omega)}\le C\left\{ \|F\|_{H^{-1}(\Omega)}+\|u_\varepsilon\|_{H^1(\Omega)}\right\},
\end{equation}
where $C$ depends only on $d$, $\mu$, and $\Omega$ (see e.g. \cite{Temam77}). 
\begin{proof}[\textbf{Proof of estimate (\ref{ineq.convergence-H1-p.Neumann})}]
Recalling the definition of $b$, we see that
\begin{equation}\label{pf1.convergence-H1-p.Neumann}
\begin{aligned}
&(\mathcal{L}_\varepsilon(v_\varepsilon))^\alpha= -\frac{\partial [p_\varepsilon-p_0]}{\partial x_\alpha}-\frac{\partial}{\partial x_i}\left(\Big[\widehat{a}_{ij}^{\alpha\beta}-a_{ij}^{\alpha\beta}(x/\varepsilon)\Big]\frac{\partial u_0^\beta}{\partial x_j}\right)\\
&\quad+\frac{\partial}{\partial x_i}\left(a_{ik}^{\alpha\gamma}(x/\varepsilon)\frac{\partial}{\partial x_k}\Big[\varepsilon \chi_j^{\gamma\beta}(x/\varepsilon)\Big]S_\varepsilon\frac{\partial \widetilde{u}_0^\beta}{\partial x_j}\right) + \varepsilon\frac{\partial}{\partial x_i}\left(a_{ik}^{\alpha\gamma}(x/\varepsilon)\chi_j^{\gamma\beta}(x/\varepsilon)S_\varepsilon \frac{\partial^2 \tilde{u}_0^\beta}{\partial x_k \partial x_j}\right)\\
%&= -\frac{\partial [p_\varepsilon-p_0-\tau_\varepsilon]}{\partial x_\alpha}-\frac{\partial}{\partial x_i}\left([\hat{a}_{ij}^{\alpha\beta}-a_{ij}^{\alpha\beta}(x/\varepsilon)]S_\varepsilon\frac{\partial \tilde{u}_0^\beta}{\partial x_j}\right)\\
%&\quad -\frac{\partial}{\partial x_i}\left([\hat{a}_{ij}^{\alpha\beta}-a_{ij}^{\alpha\beta}(x/\varepsilon)]\bigg[\frac{\partial u_0^\beta}{\partial x_j}-S_\varepsilon\frac{\partial \tilde{u}_0^\beta}{\partial x_j}\bigg]\right)\\
%&\quad +\frac{\partial}{\partial x_i}\left(a_{ij}^{\alpha\beta}(x/\varepsilon)\frac{\partial}{\partial x_j}\big[\varepsilon \chi_k^{\beta\gamma}(x/\varepsilon)\big]S_\varepsilon\frac{\partial \tilde{u}_0^\gamma}{\partial x_k}\right)\\
%&\quad + \varepsilon\frac{\partial}{\partial x_i}\left(a_{ij}^{\alpha\beta}(x/\varepsilon)\chi_k^{\beta\gamma}(x/\varepsilon)S_\varepsilon \frac{\partial^2 \tilde{u}_0^\gamma}{\partial x_j \partial x_k}\right)\\
&= -\frac{\partial [p_\varepsilon-p_0]}{\partial x_\alpha}+ \frac{\partial}{\partial x_i}\left(\Big[\widehat{a}_{ij}^{\alpha\beta}-a_{ij}^{\alpha\beta}(x/\varepsilon)\Big]\bigg[S_\varepsilon\frac{\partial \widetilde{u}_0^\beta}{\partial x_j}-\frac{\partial u_0^\beta}{\partial x_j}\bigg]\right)\\
&\quad +\varepsilon\frac{\partial}{\partial x_i}\left(a_{ik}^{\alpha\gamma}(x/\varepsilon)\chi_j^{\gamma\beta}(x/\varepsilon)S_\varepsilon \frac{\partial^2 \widetilde{u}_0^\beta}{\partial x_k \partial x_j}\right)+\frac{\partial}{\partial x_i}\left(b_{ij}^{\alpha\beta}(x/\varepsilon)S_\varepsilon\frac{\partial \widetilde{u}_0^\beta}{\partial x_j}\right).\\
\end{aligned}
\end{equation}
Using Lemma \ref{lem.dual-correctors}, we'd split the last term of RHS of (\ref{pf1.convergence-H1-p.Neumann}) into two,
\begin{equation}\label{K1-plus-K2.convergence-H1-p.Neumann}
\aligned
\frac{\partial}{\partial x_i}\left(b_{ij}^{\alpha\beta}(x/\varepsilon)S_\varepsilon\frac{\partial \tilde{u}_0^\beta}{\partial x_j}\right)&=\frac{\partial}{\partial x_i}\left(\bigg[\frac{\partial}{\partial x_k}\Big(\varepsilon \Phi_{kij}^{\alpha\beta}(x/\varepsilon)\Big)+\frac{\partial}{\partial x_\alpha}\Big(\varepsilon q_{ij}^\beta(x/\varepsilon)\Big)\bigg]S_\varepsilon\frac{\partial \widetilde{u}_0^\beta}{\partial x_j}\right)\\
&= K_1+K_2.\\
\endaligned
\end{equation}
Because of the anti-symmetry property $\Phi_{kij}^{\alpha\beta}=-\Phi_{ikj}^{\alpha\beta}$, we see that
$$
\begin{aligned}
%\frac{\partial}{\partial x_i}\left(\frac{\partial}{\partial x_k}\big(\varepsilon \Phi_{kij}^{\alpha\beta}(x/\varepsilon)\big)S_\varepsilon\frac{\partial  \tilde{u}_0^\beta}{\partial x_j}\right)
K_1&=\frac{\partial^2}{\partial x_i \partial x_k}\left(\varepsilon \Phi_{kij}^{\alpha\beta}(x/\varepsilon)S_\varepsilon\frac{\partial 
\widetilde{u_0}^\beta}{\partial x_j}\right)-\varepsilon\frac{\partial}{\partial x_i}\left(\Phi_{kij}^{\alpha\beta}(x/\varepsilon)S_\varepsilon\frac{\partial^2 \widetilde{u}_0^\beta}{\partial x_j\partial x_k}\right)\\
&=-\varepsilon\frac{\partial}{\partial x_i}\left(\Phi_{kij}^{\alpha\beta}(x/\varepsilon)S_\varepsilon\frac{\partial^2 \widetilde{u}_0^\beta}{\partial x_j\partial x_k}\right).
\end{aligned}
$$
For the second term in the RHS of (\ref{K1-plus-K2.convergence-H1-p.Neumann}), we have
\begin{equation}\label{K2-minus-K3.convergence-H1-p.Neumann}
\begin{aligned}
%\frac{\partial}{\partial x_i}\left(\frac{\partial}{\partial x_\alpha}\left(\varepsilon q_{ij}^{\beta}(x/\varepsilon)\right)S_\varepsilon\frac{\partial \tilde{u}_0^\beta}{\partial x_j}\right)%&= -\frac{\partial^2}{\partial x_\alpha \partial x_i}\left(\varepsilon q_{ij}^\beta(x/\varepsilon)\frac{\partial u_0^\beta}{\partial x_j}\right)+\frac{\partial}{\partial x_i}\left(\varepsilon q_{ij}^\beta(x/\varepsilon)\frac{\partial^2 u_0^\beta}{\partial x_\alpha \partial x_j}\right)\\
K_2&= \frac{\partial}{\partial x_\alpha}\left(\frac{\partial}{\partial x_i}\bigg[\varepsilon q_{ij}^\beta(x/\varepsilon)S_\varepsilon\frac{\partial \widetilde{u}_0^\beta}{\partial x_j}\bigg]\right)-\frac{\partial}{\partial x_i}\left(\varepsilon q_{ij}^\beta(x/\varepsilon)S_\varepsilon\frac{\partial^2 \widetilde{u}_0^\beta}{\partial x_\alpha \partial x_j}\right)\\
&=K_3-\frac{\partial}{\partial x_i}\left(\varepsilon q_{ij}^\beta(x/\varepsilon)S_\varepsilon\frac{\partial^2 \widetilde{u}_0^\beta}{\partial x_\alpha \partial x_j}\right).
\end{aligned}
\end{equation}
%Recall that $q_{ij}^\beta$ and $\pi_j^\beta$ are both periodic and satisfies
%$$
%\Delta\left(\frac{\partial q_{ij}^\beta}{\partial y_i}-\pi_j^\beta\right) =0 
%$$
%also recall that $\int_Y \pi_j^\beta =0$, which implies
%$$
%\frac{\partial q_{ij}^\beta}{\partial y_i}-\pi_j^\beta=0
%$$
In view of (\ref{eq.q-pi}), for the first term on the RHS of (\ref{K2-minus-K3.convergence-H1-p.Neumann}), we obtain 
\begin{equation}\label{pf2.convergence-H1-p.Neumann}
 %\frac{\partial}{\partial x_\alpha}\left(\frac{\partial}{\partial x_i}\bigg[\varepsilon q_{ij}^\beta(x/\varepsilon)S_\varepsilon\frac{\partial \tilde{u}_0^\beta}{\partial x_j}\bigg]\right) \\
 K_3 %& = \frac{\partial}{\partial x_\alpha}\left(\frac{\partial (q_{ij}^\beta(x/\varepsilon))}{\partial x_i}
 %S_\varepsilon\frac{\partial \widetilde{u}_0^\beta}{\partial x_j}\right)+\frac{\partial}{\partial x_\alpha}\left(\varepsilon q_{ij}^\beta(x/\varepsilon)S_\varepsilon\frac{\partial^2 \widetilde{u}^\beta_0}{\partial x_j \partial x_i}\right)\\
 = \frac{\partial}{\partial x_\alpha}\left(\pi_j^\beta(x/\varepsilon) S_\varepsilon\frac{\partial \widetilde{u}_0^\beta}{\partial x_j}\right)+\frac{\partial}{\partial x_\alpha}\left(\varepsilon q_{ij}^\beta(x/\varepsilon)S_\varepsilon\frac{\partial^2 \widetilde{u}_0^\beta}{\partial x_j \partial x_i}\right).
\end{equation}
%$$
%\begin{aligned}
%&= -\frac{\partial [p_\varepsilon-p_0-\varepsilon \tau_\varepsilon]}{\partial x_\alpha}- \frac{\partial}{\partial x_i}\left([\hat{a}_{ij}^{\alpha\beta}-a_{ij}^{\alpha\beta}(x/\varepsilon)]\bigg[\frac{\partial u_0^\beta}{\partial x_j}-S_\varepsilon\frac{\partial \tilde{u}_0^\beta}{\partial x_j}\bigg]\right)\\
%&\quad -\frac{\partial}{\partial x_i}\left(\bigg[\frac{\partial}{\partial x_k}(\varepsilon \Phi_{kij}^{\alpha\beta}(x/\varepsilon))+\frac{\partial}{\partial x_\alpha}(\varepsilon q_{ij}^{\beta}(x/\varepsilon))\bigg]S_\varepsilon\frac{\partial \tilde{u}_0^\beta}{\partial x_j}\right)\\
%&\quad+\varepsilon\frac{\partial}{\partial x_i}\left(a_{ij}^{\alpha\beta}(x/\varepsilon)\chi_k^{\beta\gamma}(x/\varepsilon)S_\varepsilon \frac{\partial^2 \tilde{u}_0^\gamma}{\partial x_j \partial x_k}\right)\\
%\end{aligned}
%$$
We have shown that
\begin{equation}\label{pf3.convergence-H1-p.Neumann}
\begin{aligned}
&\left(\mathcal{L}_\varepsilon(v_\varepsilon)\right)^\alpha+\frac{\partial}{\partial x_\alpha}\left(p_\varepsilon-p_0-\pi_j^\beta(x/\varepsilon) S_\varepsilon\frac{\partial \widetilde{u}_0^\beta}{\partial x_j}-\varepsilon q_{ij}^\beta(x/\varepsilon)S_\varepsilon\frac{\partial^2 \widetilde{u}_0^\beta}{\partial x_j \partial x_i}\right)\\
&=\varepsilon\frac{\partial}{\partial x_i}\left(\Big[a_{ij}^{\alpha\gamma}(x/\varepsilon)\chi_k^{\gamma\beta}(x/\varepsilon)-\Phi_{kij}^{\alpha\beta}(x/\varepsilon)\Big]S_\varepsilon\frac{\partial^2 \widetilde{u}_0^\beta}{\partial x_j \partial x_k}\right)\\
&\qquad \qquad 
-\varepsilon\frac{\partial}{\partial x_i}\left(q_{ij}^\beta(x/\varepsilon)S_\varepsilon\frac{\partial^2 \widetilde{u}_0^\beta}{\partial x_\alpha \partial x_j}\right)-\frac{\partial}{\partial x_i}\left(\Big[\widehat{a}_{ij}^{\alpha\beta}-a_{ij}^{\alpha\beta}(x/\varepsilon)\Big]\bigg[\frac{\partial u_0^\beta}{\partial x_j}-S_\varepsilon\frac{\partial \widetilde{u}_0^\beta}{\partial x_j}\bigg]\right).\\
\end{aligned}
\end{equation}
By applying (\ref{ineq.inherent-estimate}) to (\ref{pf3.convergence-H1-p.Neumann}), and since $\int_\Omega p_\varepsilon=\int_\Omega p_0=0$, we see that
\begin{equation}\label{pf4.convergence-H1-p.Neumann}
\begin{aligned}
&\Big\|\Big[p_\varepsilon-p_0-\pi^\varepsilon S_\varepsilon (\nabla\widetilde{u}_0)-\varepsilon q^\varepsilon S_\varepsilon(\nabla^2\widetilde{u}_0)\Big]+\average_\Omega \Big[\pi^\varepsilon S_\varepsilon (\nabla\widetilde{u}_0)+\varepsilon q^\varepsilon S_\varepsilon(\nabla^2\widetilde{u}_0)\Big]\Big\|_{L^2(\Omega)}\\
&\le \Big\|\nabla \Big[p_\varepsilon-p_0-\pi^\varepsilon S_\varepsilon (\nabla\widetilde{u}_0)-\varepsilon q^\varepsilon S_\varepsilon(\nabla^2\widetilde{u}_0)\Big]\Big\|_{H^{-1}(\Omega)}\\
%&\le C\|\nabla v_\varepsilon\|_{L^2(\Omega)}+C\varepsilon\left\|\big[|A^\varepsilon\chi^\varepsilon|+|\Phi^\varepsilon|+|q^\varepsilon|\big]S_\varepsilon(\nabla^2\widetilde{u}_0)\right\|_{L^2(\Omega)}+C\|S_\varepsilon(\nabla\widetilde{u}_0)-\nabla u_0\|_{L^2(\Omega)}\\
&\le C\|v_\varepsilon\|_{H^1(\Omega)}+C\varepsilon\Big[\|\chi\|_{L^2(Y)}+\|\Phi\|_{L^2(Y)}+\|q\|_{L^2(Y)}+1\Big]\|\widetilde{u}_0\|_{H^2(\mathbb{R}^{d})}\\
&\le C\sqrt{\varepsilon}\|u_0\|_{H^2(\Omega)},
\end{aligned}
\end{equation}
where we have used Propositions \ref{prop.steklov-difference}-\ref{prop.steklov-multiple} for the next to last inequality, (\ref{ineq.convergence-H1-u.Neumann}) and (\ref{ineq.extension}) for the last, and the constant $C$ is independent of $\varepsilon$. By Proposition \ref{prop.steklov-multiple} and (\ref{ineq.extension}), we see that
\begin{equation}\label{pf5.convergence-H1-p.Neumann}
\varepsilon\|q^\varepsilon S_\varepsilon(\nabla^2 \widetilde{u}_0)-\average q^\varepsilon S_\varepsilon(\nabla^2 \widetilde{u}_0)\|_{L^2(\Omega)} %&\le C\varepsilon\|\nabla [q^\varepsilon S_\varepsilon\nabla^2 \tilde{u}_0]\|_{H^{-1}(\Omega;\mathbb{R}^d)}\\
%&\le C\varepsilon\|q^\varepsilon S_\varepsilon\nabla^2 \tilde{u}_0\|_{L^{2}(\Omega;\mathbb{R}^d)}\le C\varepsilon\|q\|_{L^2(Y)}\|\nabla^2 \tild{u}_0\|_{L^2(\mathbb{R}^d;\mathbb{R}^{d\times d})}\\
\le C\varepsilon\|\widetilde{u}_0\|_{H^2(\mathbb{R}^d)} \le C\varepsilon\|u_0\|_{H^2(\Omega)}.
\end{equation}
By combining (\ref{pf4.convergence-H1-p.Neumann}) and (\ref{pf5.convergence-H1-p.Neumann}), we have proved that
$$
\| p_\varepsilon-p_0-\Big[\pi^\varepsilon S_\varepsilon (\nabla\widetilde{u}_0)-\average_\Omega \pi^\varepsilon S_\varepsilon (\nabla\widetilde{u}_0)\Big]\|_{L^2(\Omega)}\le C\sqrt{\varepsilon}\|u_0\|_{H^2(\Omega)}.
$$
\end{proof}

%%%%%%%%%%%%%%%%%%%%%%%%%%%%%%%%%%%%%%%%%%%%%%%%%%%%%%%%%%%%%%%%%%%%%%%%%%%%%%%%%%%%%%%%%%%%%%%%%%%%%%%%%%%%%%%%%%%%%%%%%%%%%%%%
\section{Convergence rates for $u_\varepsilon$ in $L^2$}
\setcounter{equation}{0}

To establish the sharp $O(\varepsilon)$ rate for $u_\varepsilon$ in $L^2$, we realize that
\begin{equation}\label{pf1.convergence-L2.Neumann}
\|u_\varepsilon-u_0\|_{L^2(\Omega)}\le \|v_\varepsilon\|_{L^2(\Omega)}+\varepsilon\|\chi^\varepsilon S_\varepsilon(\nabla\widetilde{u}_0)\|_{L^2(\Omega)}.
\end{equation}
By using Proposition \ref{prop.steklov-multiple} and (\ref{ineq.extension}),
$$
\|\chi^\varepsilon S_\varepsilon(\nabla \widetilde{u}_0)\|_{L^2(\Omega)}\le \|\chi\|_{L^2(Y)}\|\nabla \widetilde{u}_0\|_{L^2(\mathbb{R}^{d})}\le C\|u_0\|_{H^2(\Omega)}.
$$
Thus, the problem has been reduced to the proof of
\begin{equation}\label{pf2.convergence-L2.Neumann}
\|v_\varepsilon\|_{L^2(\Omega)}\le C\varepsilon\|u_0\|_{H^2(\Omega)},
\end{equation}
for which we'd use the duality argument.

\begin{proof}[\textbf{Proof of Theorem \ref{thm.convergence-L2.Neumann}}]

We consider the following duality problems, for any $H\in L^2(\Omega;\mathbb{R}^d)$, let $(\varphi_\varepsilon,\sigma_\varepsilon)\in H^1(\Omega;\mathbb{R}^d) \times L^2(\Omega)$ be the weak solution of the following adjoint Neumann problem of Stokes system
\begin{equation}\label{def.adjoint-Stokes.Neumann}
\left\{
\begin{aligned}
\mathcal{L}^*_\varepsilon (\varphi_\varepsilon)+\nabla \sigma_\varepsilon&=H-\average_\Omega H  & \text{ in }\Omega,\\
\text{div}(\varphi_\varepsilon)&= 0 & \text{ in }\Omega,\\
\left(\frac{\partial \varphi_\varepsilon}{\partial \nu_\varepsilon}\right)^*-\sigma_\varepsilon\cdot n&=0 & \text{ on }\partial \Omega,\\
\end{aligned}
\right.
\end{equation}
and $(\varphi_0,\sigma_0)\in H^2(\Omega;\mathbb{R}^d) \times H^1(\Omega)$ be the weak solution of its corresponding homogenized adjoint problem
\begin{equation}\label{def.adjoint-Stokes0.Neumann}
\left\{
\begin{aligned}
\mathcal{L}^*_0 (\varphi_0)+\nabla \sigma_0&=H-\average_\Omega H  & \text{ in }\Omega,\\
\text{div}(\varphi_0)&= 0 & \text{ in }\Omega,\\
\left(\frac{\partial \varphi_0}{\partial \nu_0}\right)^*-\sigma_0\cdot n &=0 & \text{ on }\partial \Omega,\\
\end{aligned}
\right.
\end{equation}
and satisfying
$$
\int_\Omega \varphi_\varepsilon=\int_\Omega \varphi_0=0, \quad \text{and }\int_\Omega \sigma_\varepsilon=\int_\Omega \sigma_0=0.
$$
Here we have used the notation: $\mathcal{L}^*_\varepsilon =-\text{div} \big(A^*(x/\varepsilon)\nabla\big)$
and $\mathcal{L}_0^* =-\text{div}\big(\widehat{A^*}\nabla \big)$.
%, while $\left(\frac{\partial \varphi_\varepsilon}{\partial \nu_\varepsilon}\right)^*$ and $\left(\frac{\partial \varphi_0}{\partial \nu_0}\right)^*$ are the corresponding conormal derivatives. 
We note that Theorem \ref{thm.convergence-H1.Neumann} continues to hold for $\mathcal{L}^{*}_\varepsilon$, as 
$A^*$ satisfies the same conditions as $A$. 
Also, by the $W^{2,2}$ estimates (\ref{ineq.energy0.Neumann}) for Stokes systems with constant coefficients in $C^{1,1}$ domains,
\begin{equation}\label{ineq.duality-energy0}
 \|\varphi_0\|_{H^2(\Omega)} +\|\sigma_0\|_{H^1(\Omega)} \le C \| H\|_{L^2(\Omega)}.
\end{equation}
As a result, we have 
\begin{equation}\label{pf3.convergence-L2.Neumann}
\|\varphi_\varepsilon-\varphi_0-\varepsilon \chi^{*\varepsilon}S_\varepsilon(\nabla\widetilde{\varphi}_0)\|_{H^1(\Omega)}\le C\sqrt{\varepsilon}\|\varphi_0\|_{H^2(\Omega)}\le C\sqrt{\varepsilon}\|H\|_{L^2(\Omega)}.
\end{equation}
where $(\chi^*,\pi^*)$ denotes the correctors associated with adjoint matrix $A^*$. Therefore through dual pairing, and integrating by parts, we obtain
\begin{equation}\label{pf4.convergence-L2.Neumann}
\begin{aligned}
\int_\Omega v_\varepsilon \bigg(H-\average_\Omega H\bigg)&= \langle \mathcal{L}_\varepsilon^*(\varphi_\varepsilon), v_\varepsilon \rangle_{H^{-1}_0(\Omega;\mathbb{R}^d)\times H^1(\Omega;\mathbb{R}^d)} +\int_\Omega v_\varepsilon \nabla \sigma_\varepsilon\\
&=a_\varepsilon(v_\varepsilon,\varphi_\varepsilon)-\int_\Omega \sigma_\varepsilon \text{ div}(v_\varepsilon).\\
\end{aligned}
\end{equation}
By (\ref{pf1.convergence-H1-u.Neumann}) and (\ref{ineq.energy.Neumann}), we know that
\begin{equation}\label{pf5.convergence-L2.Neumann}
\begin{aligned}
\left|\int_\Omega \sigma_\varepsilon \text{ div}(v_\varepsilon)\right| &\le C\|\sigma_\varepsilon\|_{L^2(\Omega)}\|\text{div}(v_\varepsilon)\|_{L^2(\Omega)} \le C\varepsilon\|H\|_{L^2(\Omega)}\|u_0\|_{H^2(\Omega)}.
\end{aligned}
\end{equation}
Since $\text{div}(\varphi_\varepsilon)=0$, then by Lemma \ref{lem.key-step},
\begin{equation}\label{pf6.convergence-L2.Neumann}
\begin{aligned}
|a_\varepsilon(v_\varepsilon,\varphi_\varepsilon)|&\le C\|u_0\|_{H^2(\Omega)}\left[\varepsilon^{1/2}\|\nabla\varphi_\varepsilon\|_{L^2(\Omega_{2\varepsilon})}+\varepsilon\|\nabla\varphi_\varepsilon\|_{L^2(\Omega)}\right]\\
&\le C\varepsilon\|u_0\|_{H^2(\Omega)}\|H\|_{L^2(\Omega)}+C\varepsilon^{1/2}\|u_0\|_{H^2(\Omega)}\|\nabla\varphi_\varepsilon\|_{L^2(\Omega_{2\varepsilon})},
\end{aligned}
\end{equation}
where we have used (\ref{ineq.energy.Neumann}) for the last inequality and $C$ is independent of $\varepsilon$. By triangle inequality,
\begin{equation}\label{pf7.convergence-L2.Neumann}
\begin{aligned}
\|\nabla\varphi_\varepsilon\|_{L^2(\Omega_{2\varepsilon})}&\le \|\nabla \big(\varphi_\varepsilon-\varphi_0-\varepsilon \chi^{*\varepsilon}S_\varepsilon(\nabla\widetilde{\varphi}_0)\big)\|_{L^2(\Omega_{2\varepsilon})}\\
&\quad +\|\nabla \varphi_0\|_{L^2(\Omega_{2\varepsilon})}+\|\nabla \big(\varepsilon \chi^{*\varepsilon}S_\varepsilon(\nabla\widetilde{\varphi}_0)\big)\|_{L^2(\Omega_{2\varepsilon})}.
\end{aligned}
\end{equation}
Directly deriving from (\ref{pf3.convergence-L2.Neumann}), we know that
\begin{equation}\label{pf8-1.convergence-L2.Neumann}
\|\nabla \big(\varphi_\varepsilon-\varphi_0-\varepsilon \chi^{*\varepsilon}S_\varepsilon(\nabla\widetilde{\varphi}_0)\big)\|_{L^2(\Omega_{2\varepsilon})} \le C\sqrt{\varepsilon}\|H\|_{L^2(\Omega)}.
\end{equation}
By using Lemma \ref{lem.near-boundary-integral} and (\ref{ineq.extension}) again, we get
\begin{equation}\label{pf8-2.convergence-L2.Neumann}
\|\nabla \varphi_0\|_{L^2(\Omega_{2\varepsilon})} \le C\left(\varepsilon\|\nabla \varphi_0\|_{H^1(\Omega)}\|\nabla \varphi_0\|_{L^2(\Omega)}\right)^{1/2}\le C\sqrt{\varepsilon}\|\varphi_0\|_{H^2(\Omega)}\le C\sqrt{\varepsilon}\|H\|_{L^2(\Omega)}.
\end{equation}
Similarly as in (\ref{pf9.key-step}), by Lemma \ref{lem.near-boundary-integral}, Proposition \ref{prop.steklov-multiple} and (\ref{ineq.extension}), we have
\begin{equation}\label{pf8-3.convergence-L2.Neumann}
\begin{aligned}
\|\nabla \big(\varepsilon \chi^{*\varepsilon}S_\varepsilon(\nabla\widetilde{\varphi}_0)\big)\|_{L^2(\Omega_{2\varepsilon})}&\le C\bigg\{\|(\nabla \chi^*)^\varepsilon S_\varepsilon (\nabla \widetilde{\varphi}_0)\|_{L^2((\partial\Omega)_{2\varepsilon})}+\varepsilon\|\widetilde{\varphi}_0\|_{H^2(\mathbb{R}^d)}\bigg\}\\
&\le C\sqrt{\varepsilon}\|\varphi_0\|_{H^2(\Omega)}\le C\sqrt{\varepsilon}\|H\|_{L^2(\Omega)}.
\end{aligned}
\end{equation}
Substituting (\ref{pf8-1.convergence-L2.Neumann}), (\ref{pf8-2.convergence-L2.Neumann}) and (\ref{pf8-3.convergence-L2.Neumann}) into (\ref{pf7.convergence-L2.Neumann}), we have proved that
\begin{equation}\label{pf9.convergence-L2.Neumann}
\begin{aligned}
\|\nabla\varphi_\varepsilon\|_{L^2(\Omega_{2\varepsilon})}\le C\sqrt{\varepsilon}\|H\|_{L^2(\Omega)}.
\end{aligned}
\end{equation}
%Therefore, by (\ref{e4.10.2}), (\ref{e4.10.3}) and (\ref{e4.11}), we see that
%\begin{equation}\label{e4.12}
%|I_{31}[\varphi_\varepsilon]|\le C\left(\sqrt{\varepsilon}\|u_0\|_{H^2(\Omega)}\right)\left(\sqrt{\varepsilon}\|H\|_{L^2(\Omega)}\right)\le C\varepsilon\|u_0\|_{H^2(\Omega)}\|H\|_{L^2(\Omega)}.
%\end{equation}
%Use the same manner on $I_{32}[\varphi_\varepsilon]$, we can show that 
%\begin{equation}\label{e4.14}
%|I_{32}[\varphi_\varepsilon]|\le C\left(\sqrt{\varepsilon}\|u_0\|_{H^2(\Omega)}\right)\left(\sqrt{\varepsilon}\|H\|_{L^2(\Omega)}\right)\le C\varepsilon\|u_0\|_{H^2(\Omega)}\|H\|_{L^2(\Omega)}.
%\end{equation}
%By combing (\ref{e4.12}) and (\ref{e4.14}) together, we know that 
%\begin{equation}\label{e4.15}
%|I_{3}[\varphi_\varepsilon]|\le C\varepsilon\|u_0\|_{H^2(\Omega)}\|H\|_{L^2(\Omega)},
%\end{equation}
%where $C$ depends only on $d$, $\mu$, and $\Omega$. 
Therefore,
\begin{equation}\label{pf10.convergence-L2.Neumann}
\begin{aligned}
|a_\varepsilon(v_\varepsilon,\varphi_\varepsilon)|&\le C\varepsilon\|u_0\|_{H^2(\Omega)}\|H\|_{L^2(\Omega)}+C\varepsilon^{1/2}\|u_0\|_{H^2(\Omega)}(\varepsilon^{1/2}\|H\|_{L^2(\Omega)})\\
&\le C\varepsilon\|H\|_{L^2(\Omega)}\|u_0\|_{H^2(\Omega)}.
\end{aligned}
\end{equation}
Hence, by using (\ref{pf10.convergence-L2.Neumann}) and (\ref{pf5.convergence-L2.Neumann}), we already proved that for any $H\in L^2(\Omega;\mathbb{R}^d)$
\begin{equation}\label{pf11.convergence-L2.Neumann}
\begin{aligned}
\bigg|\int_\Omega v_\varepsilon \big(H-\average_\Omega H\big)\bigg|&\le |a_\varepsilon(v_\varepsilon,\varphi_\varepsilon)|+\left|\int_\Omega \sigma_\varepsilon \text{ div}(\varphi_\varepsilon)\right|\\ 
&\le C\varepsilon\|H\|_{L^2(\Omega)}\|u_0\|_{H^2(\Omega)}.
\end{aligned}
\end{equation}
Therefore, since $\int_\Omega u_\varepsilon=\int_\Omega u_0=0$,
\begin{equation}\label{pf12.convergence-L2.Neumann}
\begin{aligned}
\left|\int_\Omega v_\varepsilon H\right|&\le \left|\int_\Omega v_\varepsilon \big(H-\average_\Omega H\big)\right|+\left|(\average_\Omega H)\int_\Omega v_\varepsilon \right|\\
&\le C\varepsilon\|H\|_{L^2(\Omega)}\|u_0\|_{H^2(\Omega)}+ C\left|(\average_\Omega H)\int_\Omega \varepsilon\chi^\varepsilon S_\varepsilon(\nabla\widetilde{u}_0)\right|\\
&\le C\varepsilon\|H\|_{L^2(\Omega)}\|u_0\|_{H^2(\Omega)},
\end{aligned}
\end{equation}
where we have used (\ref{pf11.convergence-L2.Neumann}) for the second inequality and Proposition \ref{prop.steklov-multiple} for the last, for any $H\in L^2(\Omega;\mathbb{R}^d)$. 
By duality, this implies
\begin{equation}\label{pf13.convergence-L2.Neumann}
\|v_\varepsilon\|_{L^2(\Omega)}\le C\varepsilon\|u_0\|_{H^2(\Omega)}.
\end{equation}
where $C$ is independent of $\varepsilon$. Therefore we have completed the proof.
\end{proof}

\noindent\textbf{Acknowledgment} The author wants to thank the anonymous referees for their very helpful comments and suggestions.
\bibliographystyle{plain}
\bibliography{Lib}

\medskip

\begin{flushleft}
Shu Gu,
Department of Mathematics,
The Florida State University,
Tallahassee, FL 32306, USA

E-mail: gu@math.fsu.edu; gushu0329@uky.edu
\end{flushleft}

\medskip

\noindent \today

\end{document}